\documentclass[english]{article}
\usepackage{lmodern}
\usepackage[T1]{fontenc}
\usepackage[latin9]{inputenc}
\usepackage{units}
\usepackage{textcomp}
\usepackage{amsmath}
\usepackage{amsthm}
\usepackage{amssymb}
\usepackage[numbers]{natbib}
\usepackage{authblk}

\makeatletter
\numberwithin{equation}{section}
\numberwithin{figure}{section}
\theoremstyle{plain}
\newtheorem{thm}{\protect\theoremname}[section]
  \theoremstyle{definition}
  \newtheorem{defn}[thm]{\protect\definitionname}
  \theoremstyle{plain}
  \newtheorem{fact}[thm]{\protect\factname}
  \theoremstyle{plain}
  \newtheorem{cor}[thm]{\protect\corollaryname}
  \theoremstyle{remark}
  \newtheorem{rem}[thm]{\protect\remarkname}
  \theoremstyle{plain}
  \newtheorem{lem}[thm]{\protect\lemmaname}
  \theoremstyle{definition}
  \newtheorem{problem}[thm]{\protect\problemname}

\makeatother

\usepackage{babel}
  \providecommand{\corollaryname}{Corollary}
  \providecommand{\definitionname}{Definition}
  \providecommand{\factname}{Fact}
  \providecommand{\lemmaname}{Lemma}
  \providecommand{\problemname}{Problem}
  \providecommand{\remarkname}{Remark}
\providecommand{\theoremname}{Theorem}

\begin{document}

\title{Well-Ordered Model Universes}
\author{Alon Navon}
\affil{Tel Aviv University}
\maketitle
\begin{abstract}
In this paper we show how to build a model of $\mathsf{ZFC}$ such
that all its inner models satisfying the Axiom of Choice are well-ordered
with respect to inclusion, and that said ordering is of arbitrary
height (including possibly $\mathbf{Ord}$ high). We do this by iterating
$\kappa$-Sacks forcing for ever-increasing $\kappa$, while showing
that such forcings do not add any unexpected intermediate inner models.
\end{abstract}

\section{Introduction}

In this section we aim to present and formalize the concept of a \emph{well-ordered
model universe}, and establish the proper axiomatic framework to work
with it. We start with an informal presentation to explain the motivation
behind this idea.
\begin{defn}
\label{def:Inner model}Let $M$ be a model of $\mathsf{ZFC}$ ($\mathsf{ZF}$).
We call a substructure $N$ of $M$ an \emph{inner model of $\mathsf{ZFC}$
}($\mathsf{ZF}$) if:
\begin{enumerate}
\item \label{enu:cond_1}$N$ is a model of $\mathsf{ZFC}$ ($\mathsf{ZF}$);
\item \label{enu:cond_2}The interpretation of $\in_{N}$ is $\in_{M}\cap N^{2}$;
\item The domain of $N$ is a transitive class of $M$;
\item $N$ has the same ordinals as $M$.
\end{enumerate}
\end{defn}
A model that satisfies only conditions \ref{enu:cond_1} and \ref{enu:cond_2}
is called a \emph{standard model}. Hence an inner model is a standard
transitive model that has the same ordinals as the base model.

Unless otherwise stated, we reserve the term \emph{inner model }to
refer exclusively to inner models of $\mathsf{ZFC}$. If we want to
discuss an inner model of $\mathsf{ZF}$ we shall refer to it explicitly
as such.

It was Kurt G�del that proved in \citep{key-7} that any model of
$\mathsf{ZF}$ has a least inner model $L$, called the constructible
universe, which is also a model of $\mathsf{ZFC+GCH}$. Thus, if $V\neq L$,
there are non-trivial inner models of $V$, and we can then partially
order them with respect to inclusion. This partial order has a unique
least element $L$ and a unique greatest element $V$, and we may
naturally enquire about its other order-theoretic properties. The
aim of this paper is to explore the consistency and implications of
the well-ordering property for this partial order, our main theorem
being the construction of a model of $\mathsf{ZFC}$ where the inner
models are not only well-ordered with respect to inclusion, but said
ordering is in fact order-isomorphic to all ordinals.

In practice, this means we construct a model $M$ of $\mathsf{ZFC}$
and a ``sequence'' of inner models $\left\langle M_{\alpha}\mid\alpha\in\mathbf{Ord}\right\rangle $
such that
\begin{enumerate}
\item $\alpha<\beta$ if and only if $M_{\alpha}\subsetneq M_{\beta}$;
\item For every inner model $N$ of $M$ there is an $\alpha\in\mathrm{\mathbf{Ord}}$
such that $N=M_{\alpha}.$
\end{enumerate}
This is a sort of tower of inner models, and to construct it we will
need to gradually extend the tower from its base, passing through
the successor case, singular limits, regular limits and finally the
class case. That will be the general path we follow, but first we
need to address a tricky part of our definition - in $\mathsf{ZF}$
we cannot formally talk about a sequence of classes, and each inner
model of $V$ is by definition a proper class. Normally, it is enough
to use the class notation as a shorthand for formulas and ignore the
particulars, but in this paper we'll also be interested in the interplay
between these classes. So in order to deal with these explicitly,
we turn to Bernays-G�del set theory, in short $\mathsf{BG}$ (or $\mathsf{BGC}$
if we add Global Choice), to serve as our axiomatic framework.

Bernays-G�del set theory, sometimes known as Von Neumann\textendash Bernays\textendash G�del
set theory ($\mathsf{NBG}$ for short), has its origins in a 1925
paper by John von Neumann \citep{key-9}, which formally introduced
classes into set theory for the first time. Von Neumann's theory employed
functions and arguments as its primitive notions, and used them to
define sets and classes. However, in the 1930s Paul Bernays reformulated
the theory by taking classes and sets as the primitive notions \citep{key-12}.
Later, while working on his proof for the relative consistency of
the Axiom of Choice \citep{key-7}, G�del significantly simplified
Bernay's theory, leading to what is now known as Bernays-G�del set
theory.

Unlike Zermelo-Fraenkel, Bernays-G�del set theory allows for two types
of objects: classes and sets. Every set is also considered a class,
and if a class is a member of another class then it is also a set.
Thus, a model $\left\langle V,\mathcal{V},\in\right\rangle $ of $\mathsf{BG}$
consists of a collection $\mathcal{V}$ of classes together with a
subcollection $V\subsetneq\mathcal{V}$ of sets and a relation $\in\subsetneq V\times\mathcal{V}$.
The axioms of this system are mostly very similar to those of $\mathsf{ZF}$,
and a formal exposition of them and their application in class forcing
can be found in \citep{key-2}. To avoid any ambiguity, we always
denote a model of $\mathsf{BG}$ in the form of a $\left\langle V,\mathcal{V},\in\right\rangle $
triplet, and a model of $\mathsf{ZF}$ as plain $V$.

In order to justify our use of $\mathsf{BG}$, we quote the following
facts about the link between it and $\mathsf{ZF}$.
\begin{fact}
$\mathsf{BG}$ ($\mathsf{BGC}$) is a conservative extension of $\mathsf{ZF}$
($\mathsf{ZFC}$).
\end{fact}
This fact was proven by Paul Cohen in \citep{key-1}. Furthermore,
Mostowski \citep{key-8} showed that every set-theoretical statement
provable in $\mathsf{ZF}$ ($\mathsf{ZFC}$) is provable in $\mathsf{BG}$
($\mathsf{BGC}$), and that if a sentence involving only set variables
is provable in $\mathsf{BG}$ ($\mathsf{BGC}$) then it is provable
in $\mathsf{ZF}$ ($\mathsf{ZFC}$) as well.
\begin{fact}
\label{fact:BG->ZF}Let $\left\langle V,\mathcal{V},\in\right\rangle $
be a model of $\mathsf{BG}$ ($\mathsf{BGC}$), where $\mathcal{V}$
is the collection of classes and $V\subseteq\mathcal{V}$ is the collection
of sets. Then by taking $V$ and $\in\cap\left(V\times V\right)$
we get a model of $\mathsf{ZF}$ ($\mathsf{ZFC}$).
\end{fact}
So if we have a model of $\mathsf{BG}$, by ``throwing away'' the
classes, we are left with a model of $\mathsf{ZF}$. What about the
other way round?
\begin{fact}
\label{fact:ZF->BG}Let $V$ be a model of $\mathsf{ZF}$. Take $\mathcal{V}$
to be the collection of all classes definable in $V$ with set parameters,
and take $\in$ to be the obvious extension of the membership relation
to $\mathcal{V}$. Then $\left\langle V,\mathcal{V},\in\right\rangle $
is a model of $\mathsf{BG}$.
\end{fact}
Note that even if $V$ is a model of $\mathsf{ZFC}$ then $\left\langle V,\mathcal{V},\in\right\rangle $
as defined above might only satisfy $\mathsf{BG}$, not $\mathsf{BGC}$.
But using class forcing one can add a uniform choice function that
is $\kappa$-closed for each $\kappa$, and so adds no new sets to
the universe. The resulting model $\left\langle V,\mathcal{V}',\in\right\rangle $
extends $\left\langle V,\mathcal{V},\in\right\rangle $, satisfies
$\mathsf{BGC}$ and has the same sets and the same restriction of
$\in$ to sets as $\left\langle V,\mathcal{V},\in\right\rangle $
(see A.1 in \citep{key-2}).

Facts \ref{fact:BG->ZF} and \ref{fact:ZF->BG} establish a useful
correspondance between models of $\mathsf{ZF}$ and $\mathsf{BG}$.
\begin{cor}
\label{cor:-is-a}$M$ is a definable proper inner model of $V$ if
and only if $M\neq V$ and $\left\langle M,\mathcal{M},\in\right\rangle \subsetneq\left\langle V,\mathcal{V},\in\right\rangle $
(as defined in fact \ref{fact:ZF->BG}).
\end{cor}
\begin{proof}
$\left(\Rightarrow\right)$ $M$ is proper and so $M\neq V$. $M$
is a definable class of $V$, therefore $M\in\mathcal{V}$, and also
every class definable in $M$ with set parameters is similarly definable
in $V$. Therefore $\mathcal{M}\subsetneq\mathcal{V}$ and $\left\langle M,\mathcal{M},\in\right\rangle \subsetneq\left\langle V,\mathcal{V},\in\right\rangle $.

$\left(\Leftarrow\right)$ $M\in\mathcal{\mathcal{M}}$ and so $M\in\mathcal{V}$,
so $M$ is a definable class in $V$. $M\subsetneq V$ and according
to fact \ref{fact:BG->ZF} is a model of $\mathsf{ZF}$.
\end{proof}
But moving to $\mathsf{BG}$ doesn't quite work out all the kinks.
In particular, $\mathsf{BG}$ doesn't allow for class membership within
another class. Therefore in order to speak of a sequence of classes
we need to abandon our standard definition of a sequence, and instead
use an alternative definition that is also suitable for classes.
\begin{defn}
\label{def:Seq of classes}Let $I$ be a class. Then we call $S\subseteq I\times V$
an $I$-indexed family of classes, and for each $i\in I$ we denote
$S_{i}=\left\{ x\mid\left(i,x\right)\in S\right\} $. If $I$ is well-ordered,
then we say $S$ is a \emph{sequence}, and $I$ is its \emph{underlying
order}.
\end{defn}
Thus, instead of having a 'family of classes' which each class is
a member of, each enumerated class is generated using straightforward
class comprehension. This perspective allows us to speak of sequences
of classes within $\mathsf{BG}$.

Now that we've set up this notion, we finally come to the primary
definition of this article:
\begin{defn}
\label{def:Well-Ordered-Model}Let $\left\langle V,\mathcal{V},\in\right\rangle $
be a model of $\mathsf{BG}$. We call a model $N\subseteq V$ of $\mathsf{ZFC}$
where all its inner models are well-ordered with respect to inclusion
a \emph{well-ordered model universe}. Formally, we postulate the existence
of a class $\mathbb{M}$ in $\left\langle V,\mathcal{V},\in\right\rangle $,
which is the sequence of all proper inner models of $N$ ordered by
inclusion. This means:
\begin{enumerate}
\item $\mathbb{M}\subseteq I\times N$;
\item \label{enu:Every-Inner}$M$ is a proper inner model of $N$ if and
only if there exists a unique $a\in I$ such that $M=M_{a}=\left\{ x\mid\left(a,x\right)\in\mathbb{M}\right\} $;
\item $I$ is a well-ordered class;
\item If $a<_{I}b$ then $\left(a,x\right)\in\mathbb{M}\rightarrow\left(b,x\right)\in\mathbb{M}$.
\end{enumerate}
\end{defn}
In summary, applying the convention that lower-case letters indicate
sets and upper-case letters indicates classes, we demand the following
be true: 
\begin{multline*}
\exists\mathbb{M}\exists I\text{ (}\mathbb{M}\subseteq I\times N\wedge\left(M\subsetneq N\text{ is an inner model }\leftrightarrow\exists!a\in I\left(M=\left\{ x\mid\left(a,x\right)\in\mathbb{M}\right\} \right)\right)\wedge\\
I\text{ is well-ordered }\wedge a<_{I}b\rightarrow\left(\left(a,x\right)\in\mathbb{M}\rightarrow\left(b,x\right)\in\mathbb{M}\right)\text{)}
\end{multline*}

Where there's any ambiguity about the base model we denote class $\mathbb{M}$
as $\mathbb{M}\left(N\right)$.

A few important remarks are in order.
\begin{rem}
\label{rem:Class--belongs}Class $\mathbb{M}\left(N\right)$ belongs
by definition to $\left\langle V,\mathcal{V},\in\right\rangle $.
It need not be definable in $N$, nor even in $V$. However in the
example that we build later in the article, $\mathbb{M}\left(V\right)$
actually \emph{will} be definable in $V$, and then instead of working
with some background model of $\mathsf{BG}$ we will take $\mathcal{V}$
to be the collection of all classes definable in $V$ with set parameters,
exactly as we did in fact \ref{fact:ZF->BG}.
\end{rem}
\begin{rem}
Our demand that $M$ be a \emph{proper} inner model is superfluous,
and only used to simplify discussion of the case $I=\mathbf{Ord}$,
when all the proper inner models are in a bijection with the ordinals.
This convention allows us to prove general theorems about all models
$M_{\alpha}$ $\alpha\in\mathbf{Ord}$, without having to constantly
special-case ``$M_{\mathbf{Ord}}$''.
\end{rem}
\begin{rem}
It is natural to ask why when defining the sequence we only demand
$I$ be well-ordered, instead of being equal to some ordinal or $\mathbf{Ord}$.
The reason for this is that we also can also consider sequences that
are \emph{longer} than the ordinals, and we don't want to unnecessarily
exclude them from the definition. We revisit this issue in the last
section of this article, but in the meanwhile we define what it means
for a well-ordered model universe to be \emph{nice}.
\end{rem}
\begin{defn}
\label{def:niceness}We call a well-ordered model universe \emph{nice}
if the underlying order of $\mathbb{M}$ is equal to some ordinal
or to $\mathbf{Ord}$.
\end{defn}
In essence, a well-ordered model universe is nice if its model tower
isn't 'too tall'. This restriction is not superficial. We shall later
see an interesting property that fails if the well-ordered model universe
isn't nice.

\begin{defn}
Let $N$ be a nice well-ordered model universe. We define its \emph{height}
to be the order-type of the underlying order of $\mathbb{M}$, so
$ht\left(N\right)=otp\left(I\right)$. In case $I=\mathbf{Ord}$,
we instead define $ht\left(N\right)=\infty$. Note that for convenience,
we designate $M_{ht\left(N\right)}=N$, even though it is \emph{not}
formally part of the sequence of proper inner models.
\end{defn}
To summarize our notational conventions, throughout this article:
\begin{enumerate}
\item Models of $\mathsf{BG}$ are always denoted as a triplet $\left\langle V,\mathcal{V},\in\right\rangle $,
whereas models of $\mathsf{ZF}$ are denoted using plain letters $V$.
\item $\mathbb{M}$ refers exclusively to the sequence of proper inner models
as defined in \ref{def:Well-Ordered-Model}.
\item $M_{\alpha}$ shall refer to the $\alpha$th inner model of sequence
$\mathbb{M}$.
\item The height of a well-ordered model universe, denoted $ht\left(V\right)$,
is the order-type of the underlying order of $\mathbb{M}$.
\item A well-ordered model universe is considered nice if the underlying
order of $\mathbb{M}$ isn't longer than $\mathbf{Ord}$.
\end{enumerate}

\section{Implications}

Now, it is time to explore some of the implications of the inner model
well-ordering property. For the rest of this section we assume $\left\langle V,\mathcal{V},\in\right\rangle \vDash\mathsf{BG}$,
$V$ is a well-ordered model universe, and $\mathbb{M}\subseteq I\times V$
is its sequence of proper inner models ordered by inclusion.
\begin{lem}
\label{lem:M0=00003DL}$M_{0}=L$.
\end{lem}
\begin{proof}
We know from G�del \citep{key-7} that $L$ is the least inner model
of $V$. Therefore, in order to be included in our hierarchy, we must
have $M_{0}=L$.
\end{proof}
This begs the question of how 'close' are $V$ and $L$, assuming
our well-ordered model universe exists. One quick observation, resulting
directly from the well-ordering of the inner models, is that $V$
cannot contain a measurable cardinal.
\begin{thm}
\label{thm:Measurable}$V\models\textrm{There is no measurable cardinal}$.
\end{thm}
\begin{proof}
Suppose to the contrary, that there is a measurable cardinal $\kappa\in V$.
Then there exists an elementary embedding $j\colon V\to M$, where
$M$ is an inner model of $V$ \citep{key-16}. Therefore $M=M_{\alpha}$
for some $\alpha\in I$. But because the embedding is elementary,
$j\left(\kappa\right)$ is measurable in $M_{\alpha}$, so there exists
an additional elementary embedding $j\left(j\right)\colon M_{\alpha}\to N$.
But this means there is an elementary embedding $V\rightarrow N$,
and so $N$ is itself an inner model of $V$ as well, so $N=M_{\beta}$
for some $\beta\in I$. However, $N\subsetneq M_{\alpha}$, and therefore
$\beta<_{I}\alpha$.

By induction, we can repeat this process and construct an infinite
descending chain of inner models. But $V$ is a well-ordered model
universe, so the inner models are well-ordered and this is impossible.
Therefore, there is no measurable cardinal in $V$.
\end{proof}
However, not only are there no measurable cardinals in a well-ordered
model universe, but we can further show it has no $0^{\sharp}$ as
well, although this is a bit less straightforward.
\begin{thm}
\label{thm:.0=000023}$V\vDash0^{\sharp}\text{ does not exist}$.
\end{thm}
\begin{proof}
Suppose to the contary, that $0^{\sharp}$ does exist. Then every
uncountable cardinal in $V$ is inaccessible in $L$ (see corollary
18.3 in \citep{key-17}).

Now remember Cohen forcing \citep{key-1}, where we use finite partial
functions, and define $\mathbb{P=}\mathsf{Fin}\left(\omega,2\right)^{L}$.
In $L$ we have $\left|\mathbb{P}\right|=\aleph_{0}$, and therefore
$\left|\mathcal{P}\left(\mathbb{P}\right)\right|=2^{\aleph_{0}}=\aleph_{1}$,
which is obviously smaller than the first inaccessible cardinal. Hence
in $V$ we have $\left|\mathcal{P}\left(\mathbb{P}\right)\right|=\aleph_{0}$,
meaning there are at most countable many dense subsets of $\mathbb{P}$,
and therefore by the Rasiowa-Sikorski lemma \citep{key-18} there
exists a generic set $G\in V\setminus L$ that intersects them all.
Hence $V\vDash\exists G\left(L\left[G\right]\supsetneq L\wedge L\left[G\right]\vDash\mathsf{ZFC}\right)$.
Therefore $L\left[G\right]$ is an inner model of $V$.

However each Cohen forcing is isomorphic to the product of two separate
Cohen forcings. Namely, for each $I_{0}\subsetneq I$ we have $\mathsf{Fin}\left(I,2\right)\cong\mathsf{Fin}\left(I_{0},2\right)\times\mathsf{Fin}\left(I\setminus I_{0},2\right)$
(see Kunen \citep{key-19} ch. VIII 2.1). So take $I=\omega$ and
$I_{0}$ the set of even natural numbers. According to the theorem
$G_{0}=G\cap\mathsf{Fin}\left(I_{0},2\right)$ is $\mathsf{Fin}\left(I_{0},2\right)$-generic
over $L$, $G_{1}=G\cap\mathsf{Fin}\left(\omega\setminus I_{0},2\right)$
is $\mathsf{Fin}\left(\omega\setminus I_{0},2\right)$-generic over
$L\left[G_{0}\right]$, and $L\left[G\right]=L\left[G_{0}\right]\left[G_{1}\right]$.
So $G_{1}\notin L\left[G_{0}\right]$, and for the same reasoning
we have $G_{0}\notin L\left[G_{1}\right]$. But as both $L\left[G_{0}\right]$
and $L\left[G_{1}\right]$ are inner models of $L\left[G\right]$
and so of $V$, either $L\left[G_{0}\right]\subsetneq L\left[G_{1}\right]$
or $L\left[G_{1}\right]\subsetneq L\left[G_{0}\right]$. Either way
we arrive at a contradiction. Therefore $0^{\sharp}$ does not exist.
\end{proof}
Now that we know $V$ cannot be too far off $L$, we may wonder if
there is perhaps a deeper connection between the two. For this we
turn to the notion of relative constructibility.

Note that there is a lot of confusion regarding its notation, so we
shall now present the notation used by Jech in \citep{key-17} and
which we adhere to.

Constructibility can be generalized in two different ways. One way
is to consider sets constructive relative to a given set $A$, resulting
in the inner model $L\left[A\right]$.

This is done by defining $def_{A}\left(M\right)=\left\{ X\subseteq M\mid X\text{ is definable over }\left(M,\in,A\cap M\right)\right\} $,
where $A\cap M$ is a unary predicate, and then defining a cumulative
hierarchy:
\begin{align*}
L_{0}\left[A\right] & =\emptyset\\
L_{\alpha+1}\left[A\right] & =def_{A}\left(L_{\alpha}\left[A\right]\right)\\
L_{\delta}\left[A\right] & =\bigcup\limits _{\alpha<\delta}L_{\alpha}\left[A\right]\text{ for limit ordinals}\\
L\left[A\right] & =\bigcup\limits _{\alpha\in\mathbf{Ord}}L_{\alpha}\left[A\right]
\end{align*}

The resulting model $L\left[A\right]$ is a model of $\mathsf{ZFC}$
(see ch. 13 in \citep{key-17}).

Another way, yields for every set $A$ the smallest inner model of
$\mathsf{ZF}$ that contains it. However, in general, this model \emph{need
not} satisfy the Axiom of Choice.

Let $T=TC\left(\left\{ A\right\} \right)$ be the transitive closure
of $A$, and define the following cumulative hierarchy: 
\begin{align*}
L_{0}\left(A\right) & =T\\
L_{\alpha+1}\left(A\right) & =def\left(L_{\alpha}\left(A\right)\right)\\
L_{\delta}\left(A\right) & =\bigcup\limits _{\alpha<\delta}L_{\alpha}\left(A\right)\text{ for limit ordinals}\\
L\left(A\right) & =\bigcup\limits _{\alpha\in\mathbf{Ord}}L_{\alpha}\left(A\right)
\end{align*}

The resulting model $L\left(A\right)$ is an inner model of $\mathsf{ZF}$,
contains $A$, and is the smallest such model.
\begin{thm}
\label{thm:L=00005BA=00005D}Let $V$ be a nice well-ordered model
universe. For all $\alpha\in I$ $M_{\alpha}=L\left[A\right]$ for
some $A\in V$.
\end{thm}
\begin{proof}
We prove this theorem by induction on $I$. It is trivially true for
$M_{0}=L=L\left[\emptyset\right]$.

For the successor stage, note that $M_{\alpha}\vDash\mathsf{AC}$.
By a theorem of Vop\v{e}nka \citep{key-20}, this means there exists
a set of ordinals $A\in M_{\alpha+1}\setminus M_{\alpha}$. Thus $M_{\alpha+1}\supseteq L\left(A\right)\supsetneq M_{\alpha}$.
However, note that for sets of ordinals $L\left(A\right)=L\left[A\right]$,
because $A\subsetneq L$. Therefore $M_{\alpha+1}\supseteq L\left[A\right]\supsetneq M_{\alpha}$.
However we know there is no model of $\mathsf{ZFC}$ strictly between
$M_{\alpha}$ and $M_{\alpha+1}$. Therefore $M_{\alpha+1}=L\left[A\right]$.

We turn to the limit stage. Let $\delta$ be a limit ordinal, and
assume the theorem was proven for all $\beta<\delta$. We want to
show that even working in $M_{\delta}$ we can enumerate all the inner
models preceding it on the model tower. Still working in $V$, for
every $\beta<\delta$ there exists a set of ordinals $C_{\beta}$
such that $M_{\beta}=L\left[C_{\beta}\right]$ (see ex. 13.27 in \citep{key-17}).
$C_{\beta}\in M_{\beta}$ and so $C_{\beta}\in M_{\delta}$, meaning
$L\left[C_{\beta}\right]^{M_{\delta}}=L\left[C_{\beta}\right]=M_{\beta}$
and therefore $M_{\beta}$ is a definable with set parameters in $M_{\delta}$.
So, working in $\left\langle M_{\delta},\mathcal{M}_{\delta},\in\right\rangle $
as defined using fact \ref{fact:ZF->BG}, for each $\beta<\delta$
$M_{\beta}\in\mathcal{M}_{\delta}$.

Also, note that any inner model of the form $L\left[A\right]$ where
$A\in M_{\delta}$ is definable using set parameters in $V$ as well,
and therefore is equal to $M_{\beta}$ for some $\beta<\delta$.

Despite having each individual model definable with set parameters
in $M_{\delta}$, we still can't be sure we can actually enumerate
all the inner models preceding $M_{\delta}$ within $\left\langle M_{\delta},\mathcal{M}_{\delta},\in\right\rangle $.
So next we define an eqivalence relation on sets of ordinals: $A\sim B\Leftrightarrow L\left[A\right]=L\left[B\right]$.
For each equivalence class $\left[A\right]$ let $r\left[A\right]$
be the sets of minimal rank in $\left[A\right]$. Obviously for each
$A$ $r\left[A\right]$ is a set, and by the induction hypothesis
and the note above there are at most $\delta$ different models of
the form $L\left[A\right]$. So $\left\{ r\left[A\right]\mid A\in V\right\} $
is a set of sets, and using the Axiom of Choice we can choose a representative
from each $r\left[A\right]$.

Next, because all of said models are equal to some $M_{\beta}$ on
the chain, we can sort the representatives according to the binary
relation $A\leq B\Leftrightarrow A\in L\left[B\right]$. The models
are well-ordered because as noted above they all belong on the tower.
We also already established that each one is definable with set parameters
in $M_{\delta}$, so what we get is a sequence of representatives
$\left\langle A_{\beta}\mid\beta<\delta\right\rangle $, which completely
enumerates the $M_{\beta}$'s for all $\beta<\delta$, and which is
defined using set parameters within $M_{\delta}$.

We define inductively two sequences $\left\langle B_{\beta}\mid\beta<\delta\right\rangle $
and $\left\langle \gamma_{\beta}\mid\beta<\delta\right\rangle $.
Let $B_{0}=\emptyset$, $\gamma_{0}=0$. For each $\beta<\delta$
define $\gamma_{\beta}=\sup\left(\bigcup\limits _{\alpha<\beta}B_{\alpha}\right)$
and $B_{\beta}=\left\{ \gamma_{\beta}+\epsilon\mid\epsilon\in A_{\beta}\right\} $.
It is clear by the definitions that $\gamma$ is strictly monotonously
increasing, and that all the $B$'s are mutually pairwise disjoint.

Let $B_{\delta}=\bigcup\limits _{\beta<\delta}B_{\beta}$. Clearly
$B_{\delta}$ is a set of ordinals. We claim $L\left[B_{\delta}\right]=M_{\delta}$.

First note that for each $\beta<\delta$ $B_{\beta}=\left(B_{\delta}\cap\gamma_{\beta+1}\right)\setminus\gamma_{\beta}$
and $A_{\beta}=\left\{ \epsilon\mid\gamma_{\beta}+\epsilon\in B_{\beta}\right\} $.
Therefore $A_{\beta}\in L\left[B_{\delta}\right]$ and so $M_{\beta}=L\left[A_{\beta}\right]\subseteq L\left[B_{\delta}\right]$.
Therefore $L\left[B_{\delta}\right]\supseteq M_{\delta}$.

On the other hand, we've already shown that $\left\langle A_{\beta}\mid\beta<\delta\right\rangle \in M_{\delta}$.
So $\left\langle B_{\beta}\mid\beta<\delta\right\rangle \in M_{\delta}$
and therefore $B_{\delta}\in M_{\delta}$, implying $L\left[B_{\delta}\right]=L\left(B_{\delta}\right)\subseteq M_{\delta}$.

We conclude that $L\left[B_{\delta}\right]=M_{\delta}$, and so the
induction is complete.
\end{proof}
It is instructive to note that we used the niceness property exactly
once, to justify how we could simultaneously choose a representative
from each $r\left[A\right]$. To do this for class-many sets would
have required the Axiom of Global Choice (see \citep{key-21}), which
as noted could be false in $\left\langle M_{\delta},\mathcal{M}_{\delta},\in\right\rangle $.
Moreover, if the underlying order was longer than $\mathbf{Ord}$,
this proof would fail because ``$\gamma_{\mathbf{Ord}}$'' would
be undefinable, as $\gamma$ is a strictly increasing sequence of
ordinals.
\begin{cor}
If $ht\left(V\right)<\infty$ then $V=L\left[A\right]$ for some $A\in V$.
\end{cor}
\begin{proof}
Use the proof above, only substitute $M_{ht\left(V\right)}$ for $V$.
\end{proof}
\begin{cor}
\label{cor:ord class}If $\left\langle V,\mathcal{V},\in\right\rangle \vDash\mathsf{BGC}$
and $I=\mathbf{Ord}$ then $V=L\left[A\right]$ for some class $A\subseteq\mathbf{Ord}$.
\end{cor}
\begin{proof}
Using Global Choice, we can choose in the limit stage class-many representatives
from all the $r\left[A\right]$'s simultaneously. Then we take \textbf{$B_{\delta}=\bigcup\limits _{\alpha\in\mathbf{Ord}}B_{\alpha}$}.
By the same arguments as in the theorem, for all $\alpha\in\mathbf{Ord}$
$L\left[B_{\delta}\right]\supseteq L\left[B_{\alpha}\right]$. But
this means $L\left[B_{\delta}\right]$ contains all the proper inner
models, hence $L\left[B_{\delta}\right]=V$.
\end{proof}
It now emerges that the models in our tower are not arbitrary at all.
They are in fact the very familiar models of the form $L\left[A\right]$.
We thus conclude that a nice well-ordered model universe $V$ is inherently
quite 'small' and 'close' to $L$, especially if $ht\left(V\right)<\infty$.

Before proceeding to the next section, it is worth noting what would
happen if instead of basing our model tower on $L$, we would base
it on some arbitrary inner model $M$. Obviously $\mathbb{M}$ wouldn't
be well-ordered anymore, so we would have to relax our definition.
We will only require that all inner models containing $M$ be on a
well-ordered chain, and that all other inner models be contained in
$M$. So below $M$ everything could be completely chaotic, but above
$M$ we would have a well-ordered tower. Now let's consider the implications.

First of all, as for theorem \ref{thm:Measurable}, this alteration
potentially allows for an infinite descending chain of models. So
let's assume that $V$ does have a measurable cardinal and $j:V\rightarrow N$
is the corresponding elementary embedding. Then $N$ must also contain
a measurable cardinal, and repeating this process, due to the well-ordering
we arrive at a model $N_{0}\subsetneq M$ after a finite number of
steps. Thus there is an elementary embedding $k:V\rightarrow N_{0}$,
and so $k\upharpoonright M$ is an elementary embedding of $M$ into
some smaller inner model. Therefore $M$ must also include a measurable
cardinal.

Theorem \ref{thm:L=00005BA=00005D} would still work as well, using
$M$ as the base for the induction. Accordingly, corollary \ref{cor:ord class}
would still hold up as well.

After analyzing the structure of well-ordered model universes, we
turn to the problem of constructing one of arbitrary height.

\section{Perfect set forcing}

In lemma \ref{lem:M0=00003DL} we proved the base of our model tower
is $L$. In this section we show how to build the first step in our
tower. Unlike the previous section, from here on we only assume that
we're working within a model of $\mathsf{ZFC}$, not $\mathsf{BG}$.
Also note that throughout this paper we follow the Israeli convention
for forcing, i.e if $p>q$ are forcing conditions, then $p$ is the
\emph{stronger} condition.

To construct the first floor in the tower, we call upon the notion
of \emph{Sacks forcing} \citep{key-22}, first invented by Gerald
Sacks, which is useful for creating minimal generic extensions. In
this section we present the original Sacks forcing and some of its
most important properties.

We assume the reader has a basic understanding of forcing. For a general
introduction to the technique of forcing, the reader may refer to
ch. VII of Kunen's book \citep{key-19}. For a more thorough exposition
and analysis of Sacks forcing, the reader may consult Geschke and
Quickert \citep{key-23}. 
\begin{defn}
Let $Seq$ denote the set of all finite binary sequences.
\begin{enumerate}
\item A \emph{tree} is a set $p\subseteq Seq$, such that for each $s\in p$
if $s\upharpoonright n\in p$ then for all $m<n$ $s\upharpoonright m\in p$.
\item If $p\subseteq Seq$ and $s\in p$, we say that $s$ \emph{splits}
in $p$ if $s^{\smallfrown}0\in p$ and $s^{\smallfrown}1\in p$.
\item If $p\subseteq Seq$ and $s$ splits in $p$ then we say $s$ is an
order $n$ splitting node if $\left|\left\{ t\subsetneq s\mid t\text{ splits in }p\right\} \right|=n$.
\item If $p\subseteq Seq$, we say $s$ is a \emph{stem} of $p$ if $s$
is a splitting node and for all $t\subsetneq s$ $t$ is not a splitting
node.
\end{enumerate}
\end{defn}
\begin{defn}
We say $p\subseteq Seq$ is a \emph{perfect} tree if:
\begin{enumerate}
\item $p$ is a tree;
\item And for every $s\in p$ there exists a splitting node $t\in p$ such
that $t\supseteq s$.
\end{enumerate}
\end{defn}
\begin{defn}
If $p$ is a perfect tree and $s\in p$ we denote $p\upharpoonright s=\left\{ t\in p\mid s\subseteq t\lor t\subseteq s\right\} $.
Plainly $p\upharpoonright s$ is perfect as well.
\end{defn}
\begin{defn}
We call $\mathbb{P}=\left\{ p\subseteq Seq\mid p\text{ is a perfect tree}\right\} $,
where $\mathbb{P}$ is ordered by reverse inclusion: $p\leq q\iff p\supseteq q$,
\emph{Sacks forcing}. Later, after we present the generalized form,
we shall refer to it as $\aleph_{0}$-Sacks forcing.
\end{defn}
We can identify the generic set $G$ with a function $f\colon\omega\to2$.
First, note that the set of perfect trees with a stem of height at
least $n$ is a dense set in $\mathbb{P}.$ Thus there are trees of
arbitrarily long finite stems in $G$. Also, if two trees $p,q$ both
have stems of height greater or equal than $n$, but the restrictions
of the stems on $n$ differ, then $ht(p\cap q)<n$, and so there is
no $r\in\mathbb{P}$ such that $r>p,q$. Hence all trees belonging
to the generic set $G$ must agree on their stems. Thus we can define
$f(n)=s(n)$, where $s$ is part of the stem of any $p\in G$. Due
to their agreement, the function $s$ is well defined, and due to
the arbitrary finite length of the stems $f$ is defined on $\omega$.
For the other direction, we may define $G=\left\{ p\in\mathbb{P}\mid\forall n\in\omega\left(f\upharpoonright n\in p\right)\right\} $
\citep{key-23}. So in essence $G$ is equivalent to a new real number,
called a \emph{Sacks real}.
\begin{lem}
\label{lem:antichain-1}$\mathsf{CH}$ implies that $\left|\mathbb{P}\right|=\aleph_{1}$
and so $\mathbb{P}$ satisfies the $\aleph_{2}$-antichain condition.
\end{lem}
\begin{proof}
We simply count the number of possible conditions. There are at most
$\aleph_{0}$ finite binary sequences, and therefore at most $2^{\aleph_{0}}$
possible trees. Assuming $\mathsf{CH}$ $2^{\aleph_{0}}=\aleph_{1}$,
and so there are at most $\aleph_{1}$ conditions and no antichains
of cardinality $\aleph_{2}$.
\end{proof}
We note that $\mathbb{P}$ does not offer much in way of closure.
It is plainly not $\aleph_{1}$-closed, as one may take any perfect
tree $p$ and build the following sequence: $\left\langle p_{n}\mid n\in\omega\right\rangle $
where $p_{0}=p$ and $p_{n+1}=p_{n}\upharpoonright s^{\frown}0$ where
$s$ is the single order $0$ splitting node of $p_{n}$. This is
obviously a sequence of perfect trees such that for all $n$ $p_{n+1}>p_{n}$,
however $\bigcap\limits _{n\in\omega}p_{n}$ has no splitting nodes
at all, and therfore is not a perfect tree.

Luckily, perfect trees offer a slightly weaker form of closure, using
the technique of \emph{fusion}.
\begin{defn}
Suppose $p,q\in\mathbb{P}$. We say $p\geq_{n}q$ if:
\begin{enumerate}
\item $p\geq q$;
\item And $s\in p$ is an order $n$ splitting in $p$ node if and only
if $s\in q$ is an order $n$ splitting node in $q$.
\end{enumerate}
\end{defn}
\begin{lem}
\label{lem:Fusion1}Fusion: Let $\left\langle p_{n}\in\mathbb{P}\mid n\in\omega\right\rangle $
be a sequence of conditions such that for all $n$ $p_{n+1}\geq_{n}p_{n}$.
Then $\bigcap\limits _{n\in\omega}p_{n}\in\mathbb{P}$.
\end{lem}
\begin{proof}
Define $\bigcap\limits _{n\in\omega}p_{n}=p_{\omega}$. We claim $p_{\omega}\in\mathbb{P}$,
meaning it's a perfect tree. Take $s\in p_{\omega}$. Let $\left|\left\{ t\subsetneq s\mid t\text{ splits in }p_{\omega}\right\} \right|=m$.
Take $p_{m+1}$. By definition $s\in p_{m+1}$. However $p_{m+1}$
is a perfect tree, and so has a splitting node of order $m+1$ above
$s$, which we denote $t\supseteq s$. But because it is an order
$m+1$ splitting node and $p_{m+2}\geq_{m+1}p_{m+1}$ we have $t\in p_{m+2}$.
By induction we get $t\in p_{\omega}$, but $s\subseteq t$ so we
found a splitting node in $p_{\omega}$ above our arbitrary $s$.
Hence $p_{\omega}$ is indeed perfect.
\end{proof}
Note that it is obvious from the chain condition that all cardinals
greater than or equal to $\aleph_{2}$ are preserved, as is of course
$\aleph_{0}$. We now complete the picture with showing $\aleph_{1}$
is preserved.
\begin{lem}
\label{lem:aleph_1-is-preserved}$\aleph_{1}$ is preserved under
$\aleph_{0}$-Sacks forcing.
\end{lem}
\begin{proof}
Assume $X$ is a countable set of ordinals in $V\left[G\right]$.
We show the existence of a set $A\in V$ countable in $V$ such that
$X\subseteq\dot{A}$. Let $\dot{F}$ be a name and let $p$ be a condition
such that $p\Vdash\dot{F}\text{ witnesses that }\dot{X}\text{ is countable}$,
that is $p\Vdash\dot{F}:\omega\rightarrow\dot{X}\text{ is surjective}$.

We now build a fusion sequence $\left\langle p_{\alpha}\mid\alpha\in\omega\right\rangle $
starting with $p_{0}=p$. Assume we defined $p_{n}$. Let $S_{n}$
be the set of all order $n$ splitting nodes of $p_{n}$. For each
$s\in S_{n}$ let $q_{s^{\frown}0},q_{s^{\frown}1}$ and $a_{s^{\frown}0},a_{s^{\frown}1}$
be such that $q_{s^{\frown}i}\geq p_{n}\upharpoonright s^{\frown}i$
and $q_{s^{\frown}i}\Vdash\dot{F}(n)=a_{s^{\frown}i}$. Let $p_{n+1}=\bigcup\limits _{s\in S_{n},i=0,1}q_{s^{\frown}i}$.

Note that the union of perfect trees is a perfect tree, and that all
splitting nodes of order $\leq n$ are preserved: if $t$ is an order
$m<n$ splitting node in $p_{n}$ then it is also a splitting node
in $q_{s^{\frown}0}$ for the $s\in S_{n}$ that is $s\supsetneq t$,
and so is in $p_{n+1}$; whereas if $t$ is an order $n$ splitting
node in $p_{n}$ then $t$ is a splitting node in $q_{t^{\frown}0}\cup q_{t^{\frown}1}$
and so is in $p_{n+1}$. Thus all splitting nodes of order $n$ are
preserved in $p_{n+1}$, and so $p_{n+1}\geq_{n}p_{n}$. Using lemma
\ref{lem:Fusion1} we get $q=\bigcap\limits _{n\in\omega}p_{n}\in\mathbb{P}$.

Now define $A=\bigcup\limits _{n\in\omega}\left\{ a_{s^{\frown}i}\mid s\in S_{n}\wedge i=0,1\right\} $.
Note that $A$ is a countable union of finite sets, hence $A$ is
countable in $V$. Now observe that $q\Vdash ran\left(\dot{F}\right)\subseteq A$.
As $\dot{F}$ is the name of the function that witnesses the countability
of $X$ this means $q\Vdash\dot{X}\subseteq A$.

In this process we built a specific $q\geq p$, so $q$ is not guaranteed
to be in the generic set $G$. However, as we found a $q\Vdash\dot{X}\subseteq A$
above any condition stronger or equal to $p$, due to density, there
is \emph{some} $r\geq p$ in $G$ such that $r\Vdash\dot{X}\subseteq A$.
Therefore $V\left[G\right]\vDash X\subseteq A$, where $A$ is countable
in $V$, which implies that $\aleph_{1}$ is preserved.
\end{proof}
\begin{thm}
\label{thm:OrgSacksMinimal}Sacks forcing produces a minimal extension
of $V$, meaning that for every model $W$ of $\mathsf{ZFC}$ if $V\subseteq W\subseteq V\left[G\right],$
then either $W=V$ or $W=V\left[G\right]$.
\end{thm}
\begin{proof}
According to theorem 15.43 of \citep{key-17}, every intermediate
model of $\mathsf{ZFC}$ is equal to $V\left[A\right]$, where $A$
is a set of ordinals. Hence it is sufficient to show that for any
set of ordinals $A$ in $V\left[G\right]$, either $V\left[A\right]=V$
or $V\left[A\right]=V\left[G\right]$.

Let $\dot{A}$ be the name of a set of ordinals in $V\left[G\right]$.
There is an ordinal $\alpha$ such that $0\Vdash\dot{A}\subseteq\alpha$,
and let $\dot{z}$ be the name of the characteristic function of $A$,
$z\colon\alpha\rightarrow2$. If $A\in V$ then obviously $V\left[A\right]=V$.
Assume then $p\in G$ is a condition that forces $\dot{A}\notin V$.

For a condition $q\in\mathbb{P}$ let $\dot{z}_{q}$ be the longest
initial segment of $\dot{z}$ that is decided by $q$, and $\gamma_{q}$
be the first ordinal for which $\dot{z}$ is undecided. For $q\geq p$
they must be well-defined, because if $q$ decides all of $\dot{z}$,
it decides all of $\dot{A}$, and then $A\in V$, in contradiction
to $p\Vdash\dot{A}\notin V$. Plainly $\gamma_{q}<\alpha$.

Mark $p_{0}=p$. Assume we've already chosen $p_{n}$. For every splitting
node $s\in S_{n}$, where $S_{n}$ is defined as in lemma \ref{lem:aleph_1-is-preserved},
let's look at $\gamma_{p_{n}\upharpoonright s}$ and conditions $p_{n}\upharpoonright s^{\frown}i$.
Suppose that for both $i=0,1$ we have $p_{n}\upharpoonright s^{\frown}i\Vdash\dot{z}\left(\gamma_{p_{n}\upharpoonright s}\right)=j$.
$\gamma_{p_{n}\upharpoonright s}$ is undecided, so take $q\geq p_{n}\upharpoonright s$
such that $q\Vdash\dot{z}\left(\gamma_{p_{n}\upharpoonright s}\right)=1-j$.
Either $q\cap p_{n}\upharpoonright s^{\frown}0\in\mathbb{P}$, or
$q\cap p_{n}\upharpoonright s^{\frown}1\in\mathbb{P}$. But $q$ and
$p_{n}\upharpoonright s^{\frown}i$ are incompatible for $i=0,1$,
hence our supposition is impossible.

Thus, if for a certain $i$ there is a $j$ such that $p_{n}\upharpoonright s^{\frown}i\Vdash\dot{z}\left(\gamma_{p_{n}\upharpoonright s}\right)=j$,
then there is some $q_{s^{\frown}\left(1-i\right)}\geq p_{n}\upharpoonright s^{\frown}\left(1-i\right)$
such that $q_{s^{\frown}\left(1-i\right)}\Vdash\dot{z}\left(\gamma_{p_{n}\upharpoonright s}\right)=1-j$,
and we take $q_{s^{\frown}i}=p_{n}\upharpoonright s^{\frown}i$ so
that $q_{s^{\frown}i}\Vdash\dot{z}\left(\gamma_{p_{n}\upharpoonright s}\right)=j$.
If there is no such $i$, then we are free to take for both $i=0,1$
$q_{s^{\frown}i}\geq p_{n}\upharpoonright s^{\frown}i$ such that
$q_{s^{\frown}i}\Vdash\dot{z}\left(\gamma_{p_{n}\upharpoonright s}\right)=i$.
The point is that in both cases we found $q_{s^{\frown}i}$ that decide
$\dot{z}\left(\gamma_{p_{n}\upharpoonright s}\right)$ in conflicting
ways for $i=0,1$.

We now take $p_{n+1}=\bigcup\limits _{s\in S_{n},i=0,1}q_{s^{\frown}i}$.
Again, exactly as in lemma \ref{lem:aleph_1-is-preserved}, we recognize
$\left\langle p_{n}\mid n\in\omega\right\rangle $ is a fusion sequence.
Thus we can take condition $q=\bigcap\limits _{n\in\omega}p_{n}$.

Let $f=\left\{ s\in q\mid\dot{z}_{q\upharpoonright s}\subseteq\dot{z}_{G}\right\} $.
This is a branch of $q$, because if $s$ is a splitting node of $q$,
then either for $i=0$ or $i=1$, but not both, $\dot{z}_{q\upharpoonright s^{\frown}i}\left(\gamma_{q\upharpoonright s}\right)=\dot{z}_{q}\left(\gamma_{q\upharpoonright s}\right)$,
hence for only one $i$ we have $s^{\frown}i\in f$. Thus $f$ is
a completely definable branch in $V\left[A\right]$, and so $f\in V\left[A\right].$

We now note that given $p$ we created a stronger condition $q$ and
so from density we can assume $q\in G$. We claim that $f$ is our
Sacks real $f$. Mark the Sacks real as $g$. If $f$ disagrees with
$g$, then because both are branches in $q$, there must be a splitting
node $s$ of $q$ where they diverge. But that would imply $\dot{z}_{f}\left(\gamma_{q\upharpoonright s}\right)\neq\dot{z}_{G}\left(\gamma_{q\upharpoonright s}\right)$
in contradiction to the definition of $f$. Thus $f=g$, $f$ is our
Sacks real, and we get $G\in V\left[A\right]$.

Therefore $V\left[G\right]\subseteq V\left[A\right]\subseteq V\left[G\right]$
and we conclude $V\left[A\right]=V\left[G\right]$.
\end{proof}
\begin{cor}
\label{cor:CH}$V\vDash\mathsf{CH}\Rightarrow V\left[G\right]\vDash\mathsf{CH}$.
\end{cor}
\begin{proof}
In the proof of theorem \ref{thm:OrgSacksMinimal} let's assume $A$
is a 'new' subset of $\aleph_{0}$, meaning we have $V\left[G\right]\vDash A\subsetneq\aleph_{0}\wedge A\notin V$.
Using fusion, we generate a perfect tree $q\in\mathbb{P}$ in the
ground model that is used to interpret $A$ according to the Sacks
real $G$.

Viewed another way, and taking $\left[q\right]$ to signify the branches
of $q$, $q$ is in fact a continuous map $q:\left[q\right]\rightarrow\mathcal{P}\left(\aleph_{0}\right)$
such that $V\left[G\right]\vDash q\left(G\right)=A$. Therefore given
$G$, there can be no two subsets of $\aleph_{0}$ $A_{1}\neq A_{2}$
that produce the same $q$.

But according to lemma \ref{lem:antichain-1} there are at most $\aleph_{1}$
conditions in $\mathbb{P}$, so there are at most $\aleph_{1}$ new
subsets of $\aleph_{0}$ in $V\left[G\right]$. Hence, $\left(2^{\aleph_{0}}=\aleph_{1}\right)^{V\left[G\right]}$.
\end{proof}

\section{$\kappa$-Sacks forcing}

In this section we show how we can extend our model tower through
the successor steps.

Na�vely we could try and repeat the Sacks forcing, hoping that no
unexpected models 'pop up' along the way. However, ultimately we desire
to iterate our forcing class-many times, so we need to be wary of
preserving the Power Set Axiom. Because each application of classical
Sacks forcing adds a real number, were we simply to iterate the forcing
class-many times, $2^{\aleph_{0}}$ would 'explode', and the resultant
model would fail to satisfy $\mathsf{ZF}$. Instead, what we need
to do is find a way to build minimal models where the subsets of each
cardinal eventually stabilize.

For this, we turn to perfect trees of height $\kappa$, via Kanamori's
extension of Sacks forcing to uncountable cardinals \citep{key-24}.
The following definitions are an almost perfect analogue to the definitions
of the previous section, except where noted otherwise.
\begin{defn}
Let $Seq=\bigcup\limits _{\alpha<\kappa}2^{\alpha}$.
\begin{enumerate}
\item A \emph{tree} is a set $p\subseteq Seq$, such that for each $s\in p$
if $s\upharpoonright\alpha\in p$ then for all $\beta<\alpha$ $s\upharpoonright\beta\in p$.
\item If $p\subseteq Seq$ and $s\in p$, we say that $s$ \emph{splits}
in $p$ if $s^{\smallfrown}0\in p$ and $s^{\smallfrown}1\in p$.
\item If $p\subseteq Seq$ and $s$ splits in $p$ then we say $s$ is an
order $\alpha$ splitting node if when we order $\left\{ t\subseteq s\mid t\text{ splits in }p\right\} $
by inclusion it is the $\alpha$th node.
\item If $p\subseteq Seq$, we say $s$ is a \emph{stem} of $p$ if $s$
is a splitting node and for all $t\subsetneq s$ $t$ is not a splitting
node.
\end{enumerate}
\end{defn}
\begin{defn}
\label{def:We-say-}We say $p\subseteq Seq$ is a \emph{perfect} tree
if:
\begin{enumerate}
\item $p$ is a tree.
\item For every $s\in p$ there exists a splitting node $t\in p$ such that
$t\supseteq s$.
\item \label{enu:height}If $\delta<\kappa$ is a limit ordinal, $s\in^{\delta}2$
and $s\upharpoonright\beta\in p$ for every $\beta<\delta$, then
$s\in p$. Intuitively '$p$ is closed'.
\item \label{enu:club}If $\delta<\kappa$ is a limit ordinal, $s\in^{\delta}2$
and for arbitrarily large $\beta<\alpha$ $s\upharpoonright\beta$
splits in $p$, then $s$ splits in $p$. Intuitively 'the splitting
nodes of $p$ are closed'.
\end{enumerate}
\end{defn}
The last two conditions are new, though it is easy to see that the
original $\aleph_{0}$-Sacks forcing satisfies them by default. Conditions
\ref{enu:height} and \ref{enu:club} are necessary to ensure the
closure property in lemma \ref{lem:-is--closed.}. Without condition
\ref{enu:height} the limit of $\omega$ trees might be empty, and
without condition \ref{enu:club} the limit might consist of just
a branch without any splitting nodes.
\begin{defn}
If $p$ is a perfect tree and $s\in p$ we denote $p\upharpoonright s=\left\{ t\in p\mid s\subseteq t\lor t\subseteq s\right\} $.
Plainly $p\upharpoonright s$ is perfect as well.
\end{defn}
\begin{defn}
We call $\mathbb{P}=\left\{ p\subseteq Seq\mid p\text{ is a perfect tree}\right\} $,
where $\mathbb{P}$ is ordered by reverse inclusion $p\leq q\iff p\supseteq q$,
$\kappa$-Sacks forcing.
\end{defn}
As before, we can identify the generic set $G$ with a function $f\colon\kappa\to2$.
There are trees with arbitrarily long stems in $G$, and these stems
must coincide on their mutual domain. Thus we can define $f(\alpha)=s(\alpha)$,
where $s$ is part of a stem for some $p\in G$. This function is
well-defined on $\kappa$. For the other direction, we may define
$G=\left\{ p\in\mathbb{P}\mid\forall\alpha<\kappa\left(f\upharpoonright\alpha\in p\right)\right\} $.
So in essence $G$ defines a new subset of $\kappa$.

Next, to achieve maximal closure in $\mathbb{P}$, we require $\kappa$
to be regular. So from here on it is assumed $\kappa$ is a regular
cardinal.
\begin{lem}
\label{lem:-is--closed.}$\mathbb{P}$ is $\kappa$-closed.
\end{lem}
\begin{proof}
Let $\delta<\kappa$, $\left\langle p_{\alpha}\mid\alpha<\delta\right\rangle $
be a sequence of increasing conditions. We claim $p=\bigcap\limits _{\alpha<\delta}p_{\alpha}\in\mathbb{P}$.
Conditions \ref{enu:height} and \ref{enu:club} of definition \ref{def:We-say-}
are trivially true in $p$. It is left to show that each node in $p$
has a splitting node above it.

Let $S$ be the splitting nodes of $p$ and for each $\alpha<\delta$
let $S_{\alpha}$ be the set of splitting nodes of $p_{\alpha}$.
$p$ is not empty because $\emptyset\in p_{\alpha}$ for all $\alpha$.

Assume $s\in p$. Then for each $\alpha<\delta$ $s\in p_{\alpha}$,
and denote $t_{\alpha}$ as the order $0$ splitting node of $p_{\alpha}\upharpoonright s$.
If $\alpha<\beta<\delta$ then $S_{\alpha}\supseteq S_{\beta}\supseteq S$.
Therefore either $\left\langle t_{\alpha}\mid\alpha<\delta\right\rangle $
stabilizes, in which case for some $\gamma$ $t_{\gamma}\in S_{\alpha}$
for all $\alpha<\delta$, and therefore $s\subseteq t_{\gamma}\in S$.
Or for each $\beta<\delta$ $\left\langle t_{\alpha}\mid\beta\leq\alpha<\delta\right\rangle \in p_{\beta}$
is an unbounded sequence of splitting nodes under $\bigcup\limits _{\alpha<\delta}t_{\alpha}$,
and therefore due to definition \ref{def:We-say-} condition \ref{enu:club}
$\bigcup\limits _{\alpha<\delta}t_{\alpha}\in S_{\beta}$. But that
means $s\subsetneq\bigcup\limits _{\alpha<\delta}t_{\alpha}\in S$.

Either way, we found a splitting node above an arbitrary $s\in p$,
and therefore $p$ is perfect.
\end{proof}
\begin{lem}
\label{lem:antichain}If $2^{<\kappa}=\kappa$ and $2^{\kappa}=\kappa^{+}$
then $\left|\mathbb{P}\right|=\kappa^{+}$ and so $\mathbb{P}$ satisfies
the $\kappa^{++}$-antichain condition.
\end{lem}
\begin{proof}
We simply count the number of possible conditions. There are at most
$\kappa$ binary sequences of length $<\kappa$, and therefore at
most $2^{\kappa}=\kappa^{+}$ possible trees. So there are at most
$\kappa^{+}$ conditions and no antichains of cardinality $\kappa^{++}$.
\end{proof}
We now extend the technique of fusion to this forcing.
\begin{defn}
\label{def:fusion}Suppose $p,q\in\mathbb{P}$. We say $p\geq_{\alpha}q$
if:
\begin{enumerate}
\item $p\geq q$;
\item And for all $\beta\leq\alpha$, $s\in p$ is an order $\beta$ splitting
in $p$ node if and only if $s\in q$ is an order $\beta$ splitting
node in $q$.
\end{enumerate}
\end{defn}
Note that because of the closure of the splitting nodes (definition
\ref{def:We-say-} condition \ref{enu:club}), if $\delta$ is a limit
ordinal $p\geq_{\delta}q\Leftrightarrow\forall\alpha<\delta\left(p\geq_{\alpha}q\right)$.
\begin{lem}
\label{lem:Fusion:-Let-}Fusion: Let $\left\langle p_{\alpha}\in\mathbb{P}\mid\alpha<\kappa\right\rangle $
be a sequence of conditions such that for all $\alpha$ $p_{\alpha+1}\geq_{\alpha}p_{\alpha}$,
and for $\delta$ a limit ordinal $p_{\delta}=\bigcap\limits _{\alpha<\delta}p_{\alpha}$.
Then $\bigcap\limits _{\alpha<\kappa}p_{\alpha}\in\mathbb{P}$.
\end{lem}
\begin{proof}
Define $\bigcap\limits _{\alpha<\kappa}p_{\alpha}=p_{\kappa}$. We
claim $p_{\kappa}\in\mathbb{P}$, meaning it is a perfect tree. Conditions
\ref{enu:height} and \ref{enu:club} of definition \ref{def:We-say-}
are trivially true in $p_{\kappa}$. It is left to show that each
node in $p_{\kappa}$ has a splitting node above it.

Take $s\in p_{\kappa}$. Let $\beta$ be the order type of $\left\{ t\subsetneq s\mid t\text{ splits in }p_{\kappa}\right\} $
ordered by inclusion. Take $p_{\beta+1}$. By definition $s\in p_{\beta+1}$.
However $p_{\beta+1}$ is a perfect tree, and so has a splitting node
of order $\beta+1$ above $s$, which we denote $t\supseteq s$.

Now we proceed by transfinite induction. Assume $t$ is an order $\beta+1$
splitting node in $p_{\epsilon}$, where $\epsilon>\beta$. Then $p_{\epsilon+1}\geq_{\epsilon}p_{\epsilon}$,
so $t$ is also an order $\beta+1$ splitting node in $p_{\epsilon+1}$.
Let $\delta<\kappa$ be a limit ordinal, such that for all $\epsilon$
with $\beta<\epsilon<\delta$ $t$ is an order $\beta+1$ splitting
node in $p_{\epsilon}$. Then by definition \ref{def:We-say-} condition
\ref{enu:club} $t$ is an order $\beta+1$ splitting node in $p_{\delta}$.

Therefore by induction $t\supseteq s$ is a splitting node in $p_{\kappa}$.
We found a splitting node in $p_{\kappa}$ above our arbitrary $s$.
Hence $p_{\kappa}$ is indeed perfect.
\end{proof}
Note that it is obvious from the chain condition that all cardinals
greater than or equal to $\kappa^{++}$ are preserved. On the other
hand, due to closure all cardinals less than or equal to $\kappa$
are preserved. So to complete the picture we must show $\kappa^{+}$
is preserved.
\begin{lem}
\label{lem:pres}If $2^{<\kappa}=\kappa$ then $\kappa^{+}$ is preserved
under $\kappa$-Sacks forcing.
\end{lem}
\begin{proof}
This proof closely mirrors the proof of lemma \ref{lem:aleph_1-is-preserved}.

Assume $X$ is a set of ordinals in $V\left[G\right]$, such that
$\left(\left|X\right|=\kappa\right)^{V\left[G\right]}$. We show the
existence of a set $A\in V$ of cardinality $\kappa$ in $V$ such
that $X\subseteq\dot{A}$. Let $\dot{F}$ be a name and let $p$ be
a condition such that $p\Vdash\dot{F}:\kappa\rightarrow\dot{X}\text{ is surjective}$.

We now build a fusion sequence $\left\langle p_{\alpha}\mid\alpha<\kappa\right\rangle $
with $p_{0}=p$. Assume we defined $p_{\alpha}$. Let $S_{\alpha}$
be the set of all order $\alpha$ splitting nodes of $p_{\alpha}$.
For each $s\in S_{\alpha}$ let $q_{s^{\frown}0},q_{s^{\frown}1}$
and $a_{s^{\frown}0},a_{s^{\frown}1}$ be such that $q_{s^{\frown}i}\geq p_{\alpha}\upharpoonright s^{\frown}i$
and $q_{s^{\frown}i}\Vdash\dot{F}(\alpha)=a_{s^{\frown}i}$. Let $p_{\alpha+1}=\bigcup\limits _{s\in S_{n},i=0,1}q_{s^{\frown}i}$.

All splitting nodes of order $\leq\alpha$ are preserved: if $t$
is an order $\beta<\alpha$ splitting node in $p_{\alpha}$ then it
is also a splitting node in $q_{s^{\frown}0}$ for the $s\in S_{\alpha}$
that is $s\supsetneq t$, and so in $p_{\alpha+1}$; whereas if $t$
is an order $\alpha$ splitting node in $p_{\alpha}$ then $t$ is
a splitting node in $q_{t^{\frown}0}\cup q_{t^{\frown}1}$ and so
is in $p_{\alpha+1}$. Thus all splitting nodes of order $\alpha$
are preserved in $p_{\alpha+1}$, and so $p_{\alpha+1}\geq_{\alpha}p_{\alpha}$.

In the limit case we define $p_{\delta}=\bigcap\limits _{\alpha<\delta}p_{\alpha}$.
Thus we have a fusion sequence, and using lemma \ref{lem:Fusion:-Let-}
we get $q=\bigcap\limits _{\alpha<\kappa}p_{\alpha}\in\mathbb{P}$.

Now define $A=\bigcup\limits _{\alpha<\kappa}\left\{ a_{s^{\frown}i}\mid s\in S_{\alpha}\wedge i=0,1\right\} $.
Note that $A$ is a union of $\kappa$ sets of at most $2^{<\kappa}=\kappa$
cardinality, so $\left|A\right|=\kappa$. Now observe that $q\Vdash ran\left(\dot{F}\right)\subseteq A$.
As $\dot{F}$ is the name of the function that witnesses the cardinality
of $X$, this means $q\Vdash\dot{X}\subseteq A$.

Although we built a specific $q\geq p$, due to density there is \emph{some}
$r\geq p$ in $G$ such that $r\Vdash\dot{X}\subseteq A$. Therefore
$V\left[G\right]\vDash X\subseteq A$, where $A$ is of cardinality
$\kappa$ in $V$, which implies that $\kappa^{+}$ is preserved.
\end{proof}
\begin{thm}
\label{thm:-Sacks-forcing-produces}$\kappa$-Sacks forcing produces
a minimal extension of $V$, such that for every model $W$ of $\mathsf{ZFC}$
if $V\subseteq W\subseteq V\left[G\right],$ then either $W=V$ or
$W=V\left[G\right]$.
\end{thm}
\begin{proof}
This proof closely mirrors the proof of theorem \ref{thm:OrgSacksMinimal},
and so is given here in a more concise form.

We show that for any set of ordinals $A\in V\left[G\right]$ either
$V\left[A\right]=V$ or $V\left[A\right]=V\left[G\right]$. Let $\dot{A}$
be the name of a set of ordinals in $V\left[G\right]$, and let $\dot{z}$
be the name of its characteristic function. Assume $p\in G$ forces
$\dot{A}\notin V$.

For a condition $q\in\mathbb{P}$ let $\dot{z}_{q}$ be the longest
initial segment of $\dot{z}$ that is decided by $q$, and $\gamma_{q}$
be the first ordinal for which $\dot{z}$ is undecided.

Mark $p_{0}=p$. Assume we've already chosen $p_{\alpha}$. For every
splitting node $s\in S_{\alpha}$, where $S_{\alpha}$ is defined
as in lemma \ref{lem:pres}, let's look at $\gamma_{p_{\alpha}\upharpoonright s}$
and conditions $p_{\alpha}\upharpoonright s^{\frown}i$. If for a
certain $i$ there is a $j$ such that $p_{\alpha}\upharpoonright s^{\frown}i\Vdash\dot{z}\left(\gamma_{p_{\alpha}\upharpoonright s}\right)=j$
and we take $q_{s^{\frown}i}=p_{\alpha}\upharpoonright s^{\frown}i$,
there will be $q_{s^{\frown}\left(1-i\right)}\geq p_{\alpha}\upharpoonright s^{\frown}\left(1-i\right)$
such that $q_{s^{\frown}\left(1-i\right)}\Vdash\dot{z}\left(\gamma_{p_{\alpha}\upharpoonright s}\right)=1-j$.
If there is no such $i$, then we are free to take for both $i=0,1$
$q_{s^{\frown}i}\geq p_{\alpha}\upharpoonright s^{\frown}i$ such
that $q_{s^{\frown}i}\Vdash\dot{z}\left(\gamma_{p_{\alpha}\upharpoonright s}\right)=i$.
Either case we found $q_{s^{\frown}i}$ that decide $\dot{z}\left(\gamma_{p_{\alpha}\upharpoonright s}\right)$
in conflicting ways for $i=0,1$.

We now take $p_{\alpha+1}=\bigcup\limits _{s\in S_{n},i=0,1}q_{s^{\frown}i}$,
and for limit ordinals $p_{\delta}=\bigcap\limits _{\alpha<\delta}p_{\alpha}$.
Exactly as in lemma \ref{lem:pres}, $\left\langle p_{\alpha}\mid\alpha<\kappa\right\rangle $
is a fusion sequence and we can take $q=\bigcap\limits _{\alpha<\kappa}p_{\alpha}$.
Let $f=\left\{ s\in q\mid\dot{z}_{q\upharpoonright s}\subseteq\dot{z}_{G}\right\} $.
$f$ is a branch of $q$, and using $A$ is completely definable,
so $f\in V\left[A\right].$

Due to density we may assume $q\in G$, in which case $f$ is actually
our new function $\kappa\to2$, which we identify with $G$. Thus
we get $G\in V\left[A\right]$, and so $V\left[A\right]=V\left[G\right]$.
\end{proof}
\begin{cor}
\label{cor:.2^<k}$V\vDash2^{\kappa}=\kappa^{+}\wedge2^{<\kappa}=\kappa\Rightarrow V\left[G\right]\vDash2^{\kappa}=\kappa^{+}$.
\end{cor}
\begin{proof}
Again, in direct analogy to corollary \ref{cor:CH}. In the proof
of theorem \ref{thm:-Sacks-forcing-produces} let's assume $A$ is
a 'new' subset of $\kappa$, meaning we have $V\left[G\right]\vDash A\subsetneq\kappa\wedge A\notin V$.
Using fusion, we generate a perfect tree $q\in\mathbb{P}$ in the
ground model that is used to interpret $A$ according to $G$.

Taking $\left[q\right]$ to signify the branches of $q$, we can view
$q$ as a mapping $q:\left[q\right]\rightarrow\mathcal{P}\left(\kappa\right)$
between the branches of $q$ and subsets of $\kappa$. Our construction
method for $q$ implies that $V\left[G\right]\vDash q\left(G\right)=A$.
Therefore for a given $G$, there cannot be two subsets of $\kappa$
$A_{1}\neq A_{2}$ that produce the same $q$.

But according to lemma \ref{lem:antichain} there are at most $\kappa^{+}$
conditions in $\mathbb{P}$, so there can be at most $\kappa^{+}$
new subsets of $\kappa$ in $V\left[G\right]$. Hence, $\left(2^{\kappa}=\kappa^{+}\right)^{V\left[G\right]}$.
\end{proof}
After proving minimality and preservation of cardinals we conclude
this section with showing that $\kappa$-Sacks forcing preserves $\mathsf{GCH}$
above $\kappa$.
\begin{lem}
\label{lem:GCHpres}$V\vDash\forall\lambda\geq\kappa\left(2^{\lambda}=\lambda^{+}\right)\Rightarrow V\left[G\right]\vDash\forall\lambda>\kappa\left(2^{\lambda}=\lambda^{+}\right)$.
\end{lem}
\begin{proof}
For the preservation of $\mathsf{GCH}$ above $\kappa$, we turn to
the notion of a \emph{nice name} (see ch. VII definition 5.11 in \citep{key-19})\emph{.}
A name for a subset of $\sigma\in V^{\mathbb{P}}$ is considered nice
if it is of the form $\bigcup\left\{ \left\{ \pi\right\} \times A_{\pi}\mid\pi\in\mathrm{dom}\left(\sigma\right)\right\} $,
where each $A_{\pi}$ is an antichain in $\mathbb{P}$. Every subset
has a nice name.

According to lemma \ref{lem:antichain} there are at most $\kappa^{+}$
conditions, and therefore at most $\kappa^{+}$ elements in an antichain.
Meaning, there are at most $\left(\kappa^{+}\right)^{\kappa^{+}}=\kappa^{++}$
different possible antichains.

Hence for a given cardinal $\lambda$ there are at most $\left(\kappa^{++}\right)^{\lambda}$
different nice names for subsets of $\dot{\lambda}$. Thus for $\lambda\geq\kappa^{++}$
we have $\left(\kappa^{++}\right)^{\lambda}=2^{\lambda}=\lambda^{+}$,
meaning $\left(2^{\lambda}\leq\lambda^{+}\right)^{V\left[G\right]}$.
But of course this means $V\left[G\right]\Vdash2^{\lambda}=\lambda^{+}$.
For the case $\lambda=\kappa^{+}$ we derive this instead from $\lambda^{+}=\kappa^{++}=\left(\kappa^{++}\right)^{\kappa^{+}}=\left(\kappa^{++}\right)^{\lambda}$.

Thus $V\left[G\right]\vDash\forall\lambda>\kappa\left(2^{\lambda}=\lambda^{+}\right)$,
as required.
\end{proof}
We now briefly summarize the attributes we demanded from $\kappa$
for the forcing notion and the above theorems to make sense:
\begin{enumerate}
\item $\kappa$ needed to be regular, for the closure to work (lemma \ref{lem:-is--closed.}).
\item \label{enu:sum2}$2^{<\kappa}=\kappa$ is necessary to preserve $\kappa^{+}$
(lemma \ref{lem:pres}).
\item \label{enu:sum3}$2^{<\kappa}=\kappa$ and $2^{\kappa}=\kappa^{+}$
are necessary for the antichain condition (lemma \ref{lem:antichain}).
\end{enumerate}
All of these conditions are necessary to prove the preservation of
cardinals by $\kappa$-Sacks forcing. All of them are automatically
true if $\kappa$ strongly inaccessible. However, for our construction
we don't want to rely on the existence of large cardinals. Notably,
both conditions \ref{enu:sum2} and \ref{enu:sum3} are also implied
by $\mathsf{GCH}$, so in the next section corollary \ref{cor:.2^<k}
and lemma \ref{lem:GCHpres} will serve us in maintaining enough of
$\mathsf{GCH}$ to make the forcing iteration work.

\section{$\kappa$-Sacks iteration}

After defining individual $\kappa$-Sacks forcing, it is time to stitch
everything together. We now define the forcing iteration that will
enable us to build a model tower through limit ordinals, and up to
arbitrary height.

For a general introduction to iterated forcing the reader can refer
to Shelah \citep{key-29}.

For the rest of this section, let $L$ be our base model, and let
$\zeta$ be the height of the model tower that we wish to build.
\begin{defn}
\label{def:iteration}For $\alpha\leq\zeta$, define the forcing iteration
$\mathbb{P}_{\alpha}$ as follows:
\begin{enumerate}
\item Let $\mathbb{\mathbb{\dot{Q}}_{\alpha}}$ be trivial if $\alpha$
is a limit ordinal, and the name of $\aleph_{\alpha}$-Sacks forcing
in $V^{\mathbb{P}_{\alpha}}$ otherwise.
\item $\mathbb{P}_{\alpha+1}=\mathbb{P}_{\alpha}\star\mathbb{\dot{Q}}_{\alpha}$.
\item At limit stages we use full support, i.e if $\delta$ is a limit ordinal
then $p\in\mathbb{P}_{\delta}\Leftrightarrow\forall\alpha<\delta\left(p\upharpoonright\alpha\in\mathbb{P}_{\alpha}\right)$.
\end{enumerate}
\end{defn}
\begin{defn}
Denote:
\begin{enumerate}
\item $M_{0}=L$.
\item $G_{\alpha}$ as the generic set in partial order $\mathbb{P}_{\alpha}$
over $M_{0}$.
\item $M_{\alpha}=M_{0}\left[G_{\alpha}\right]$.
\item $G_{\alpha,\beta}$ as the generic set in partial order $\nicefrac{\mathbb{P}_{\beta}}{G_{\alpha}}$
over $M_{0}\left[G_{\alpha}\right]$, so that $M_{0}\left[G_{\alpha}\right]\left[\dot{G}_{\alpha,\beta}\right]=M_{0}\left[G_{\beta}\right]$.
\end{enumerate}
\end{defn}
\begin{lem}
\label{lem:IterationIsClosed}For every $\alpha<\beta$ $\nicefrac{\mathbb{P}_{\beta}}{G_{\alpha}}$
is $\aleph_{\alpha}$-closed. If $\alpha$ is a limit ordinal, then
$\nicefrac{\mathbb{P}_{\beta}}{G_{\alpha}}$ is $\aleph_{\alpha+1}$-closed.
\end{lem}
\begin{proof}
Every coordinate of the forcing $\nicefrac{\mathbb{P}_{\beta}}{G_{\alpha}}$
is either trivial or $\aleph_{\gamma}$-Sacks forcing for $\gamma\geq\alpha$.
According to lemma \ref{lem:-is--closed.} each coordinate is therefore
$\aleph_{\alpha}$-closed.

By definition \ref{def:iteration} we use full support, and therefore
the iteration $\nicefrac{\mathbb{P}_{\beta}}{G_{\alpha}}$ as a whole
is $\aleph_{\alpha}$-closed.

If $\alpha$ is limit ordinal then $\mathbb{Q}_{\alpha}$ is trivial,
and so $\nicefrac{\mathbb{P}_{\beta}}{G_{\alpha}}=\nicefrac{\mathbb{P}_{\beta}}{G_{\alpha+1}}$
which is $\aleph_{\alpha+1}$-closed.
\end{proof}
We now show that all cardinals are preserved throughout the entire
forcing iteration.
\begin{defn}
Let $OK\left(\alpha\right)$ denote that:
\begin{enumerate}
\item \label{enu:samecard}$M_{\alpha}$ has the same cardinals as $M_{0}$;
\item \label{enu:2^<}$M_{\alpha}\vDash2^{<\aleph_{\alpha}}=\aleph_{\alpha}$;
\item \label{enu:GCHabove}$M_{\alpha}\vDash\forall\lambda\geq\aleph_{\alpha}\left(2^{\lambda}=\lambda^{+}\right)$.
\end{enumerate}
\end{defn}
\begin{lem}
\label{lem:OK0}$OK\left(0\right)$.
\end{lem}
\begin{proof}
$M_{0}=L$ and so satisfies $\mathsf{GCH}$.
\end{proof}
\begin{lem}
\label{lem:OK-succ}If $OK\left(\alpha\right)$ then $OK\left(\alpha+1\right)$.
\end{lem}
\begin{proof}
If $\mathbb{Q}_{\alpha}$ is trivial then $M_{\alpha+1}=M_{\alpha}$
and conditions \ref{enu:samecard} and \ref{enu:GCHabove} are trivially
true for $\alpha+1$. As for condition \ref{enu:2^<}, we know $M_{\alpha}\vDash2^{\aleph_{\alpha}}=\aleph_{\alpha+1}$
and so $M_{\alpha+1}\vDash2^{<\aleph_{\alpha+1}}=2^{\aleph_{\alpha}}=\aleph_{\alpha+1}$
as required.

Otherwise, $OK\left(\alpha\right)$ implies $M_{\alpha}$ has the
same cardinals as $M_{0}$, $M_{\alpha}\vDash2^{<\aleph_{\alpha}}=\aleph_{\alpha}$
and $M_{\alpha}\vDash\forall\lambda\geq\aleph_{\alpha}\left(2^{\lambda}=\lambda^{+}\right)$.
According to lemmas \ref{lem:-is--closed.}, \ref{lem:antichain}
and \ref{lem:pres} $M_{\alpha+1}$ has the same cardinals as $M_{\alpha}$,
and therefore the same as $M_{0}$. According to corollary \ref{cor:.2^<k}
$M_{\alpha+1}\vDash2^{<\aleph_{\alpha+1}}=2^{\aleph_{\alpha}}=\aleph_{\alpha+1}$.
And according to lemma \ref{lem:GCHpres} $M_{\alpha+1}\vDash\forall\lambda\geq\aleph_{\alpha+1}\left(2^{\lambda}=\lambda^{+}\right)$.
So $OK\left(\alpha+1\right)$ is true.
\end{proof}
\begin{lem}
\label{lem:SameCardinalsInLimit}Let $\delta\leq\zeta$ be a limit
ordinal. If for all $\alpha<\delta$ $OK\left(\alpha\right)$, then
$M_{\delta}$ has the same cardinals as $M_{0}$.
\end{lem}
\begin{proof}
For any $\alpha<\delta$ we have $M_{\delta}=M_{\alpha}\left[G_{\alpha,\delta}\right]$.
But $\nicefrac{\mathbb{P}_{\delta}}{G_{\alpha}}$ is $\aleph_{\alpha}$-closed
per lemma \ref{lem:IterationIsClosed}. Thus, all cardinals less than
or equal $\aleph_{\alpha}$ in $M_{\alpha}$ are preserved in $M_{\delta}$.
But $OK\left(\alpha\right)$ implies $M_{\alpha}$ has the same cardinals
as $M_{0}$, so all cardinals less than or equal $\aleph_{\alpha}$
in $M_{0}$ are preserved in $M_{\delta}$.

As this is true for all cardinals less than $\aleph_{\delta}$, it
is true for $\aleph_{\delta}$ itself.

Also note $\left|\mathbb{P_{\delta}}\right|=\prod_{\alpha<\delta}2^{\aleph_{\alpha}}=\prod_{\alpha<\delta}\aleph_{\alpha+1}\leq\aleph_{\delta}^{\left|\delta\right|}=2^{\aleph_{\delta}}=\aleph_{\delta+1}$.
Therefore $\mathbb{P}_{\delta}$ is $\aleph_{\delta+2}$-c.c, and
$M_{\delta}$ preserves all cardinals greater or equal to $\aleph_{\delta+2}$.

It remains to be proven that $\aleph_{\delta+1}$ is preserved.

First, assume $\aleph_{\delta}$ is singular and suppose that the
iteration does collapse it. Denote $cf\left(\aleph_{\delta+1}^{M_{0}}\right)^{M_{\delta}}=\aleph_{\gamma}$.
The collapse of the cardinal implies $\aleph_{\gamma}\le\aleph_{\delta}$,
and due to the latter's singularity $\aleph_{\gamma}<\aleph_{\delta}$,
meaning there is a new set of ordinals $A\in M_{\delta}$, such that
$\left|A\right|=\aleph_{\gamma}$. However according to lemma \ref{lem:IterationIsClosed}
$\nicefrac{\mathbb{P}_{\delta}}{G_{\gamma+1}}$ is $\aleph_{\gamma+1}$-closed,
so no sets of ordinals of cardinality $\aleph_{\gamma}$ are added
when forcing $M_{\delta}=M_{\gamma+1}\left[G_{\gamma+1,\delta}\right]$.
Thus $A\in M_{\gamma+1}$, which implies $M_{\gamma+1}\vDash cf\left(\aleph_{\delta+1}^{M_{0}}\right)\le\aleph_{\gamma}<\aleph_{\delta+1}$,
in contradiction to $\aleph_{\delta+1}$ being preserved in $M_{\gamma+1}$.
Therefore $\aleph_{\delta+1}$ must be preserved as well.

For the case $\aleph_{\delta}$ is regular we proceed with a variation
of the argument used in lemma \ref{lem:pres}.

Assume $\aleph_{\delta}$ is regular, meaning $\aleph_{\delta}=\delta$,
and that $F:\delta\rightarrow\mathbf{Ord}$ is a function in $M_{\delta}$.
We show the existence of a set $A\in M_{0}$ of cardinality $\delta$,
such that in $M_{\delta}$ $ran\left(F\right)\subseteq\check{A}$.
Let $p\in\mathbb{P}_{\delta}$ be a condition such that $p\Vdash\dot{F}:\delta\rightarrow\mathbf{Ord}$.

For any condition $q\in\mathbb{P}_{\delta}$, let $\dot{q}^{\alpha}$
denote the $\alpha$ coordinate of $q$, and similarly let $q^{<\alpha}$
denote the first $\alpha$ coordinates, and $\dot{q}^{>\alpha}$ denote
the name of all higher coordinates. For consistency, if we discuss
a condition $q\in\nicefrac{\mathbb{P}_{\delta}}{H_{\beta}}$ where
$H_{\beta}$ is a generic set in $\mathbb{P}_{\beta}$, then we fix
the first coordinate to be $q^{\beta}$.

Inductively we are going to build an increasing sequence of conditions
$\left\langle p\left(\alpha\right)\in\mathbb{P}_{\delta}\mid\alpha<\delta\right\rangle $.
Each coordinate is also going to be built inductively.

Start with the first coordinate. Let $p_{0}=p$, and assume we've
defined $p_{n}$ for $n<\omega$. Let $S_{n}$ be the set of all order
$n$ splitting nodes of $p_{n}^{0}$. For each $s\in S_{n}$ let $q_{s^{\frown}0},q_{s^{\frown}1}$
and $a_{s^{\frown}0},a_{s^{\frown}1}$ be such that $q_{s^{\frown}i}\geq\left\langle p_{n}^{0}\upharpoonright s^{\frown}i,p_{n}^{>0}\right\rangle $
and $q_{s^{\frown}i}\Vdash\dot{F}(n)=a_{s^{\frown}i}$. Let $p_{n+1}^{0}=\bigcup\limits _{s\in S_{n},i=0,1}q_{s^{\frown}i}^{0}$
be the amalgamation of the $q_{s^{\frown}i}^{0}$'s, just like we
did in the proof of lemma \ref{lem:pres}. For the rest of the coordinates,
we define $p_{n+1}^{>0}$ with accordance to the path taken in the
first coordinate. Meaning that $q_{s^{\frown}i}^{0}$ forces $p_{n+1}^{>0}=q_{s^{\frown}i}^{>0}$.

From the way we defined the $q_{s^{\frown}i}$'s it is clear $p_{n+1}^{>0}\geq p_{n}^{>0}$,
and so $p_{n+1}\geq p_{n}$. Also, all splitting nodes of order $\leq n$
are preserved in the first coordinate, and therefore $p_{n+1}^{0}\geq_{n}p_{n}^{0}$.
Thus what we have in the first coordinate is a classical fusion sequence,
and in the rest of the coordinates an increasing sequence. So thanks
to lemmas \ref{lem:Fusion1} and \ref{lem:IterationIsClosed} we can
conclude $p_{\omega}=\bigcap\limits _{n<\omega}p_{n}\in\mathbb{P}_{\delta}$.

We define $A_{0}=\left\{ a_{s^{\frown}i}\mid s\in S_{n}\wedge i=0,1\right\} $.
Obviously $\left|A_{0}\right|=\aleph_{0}$.

In essence, we used a fusion argument on the first coordinate to create
$p_{\omega}$, which is a sort of 'decision tree' for the first $\omega$
values of $\dot{F}$. We set $p\left(0\right)=p_{\omega}$. Next,
we are going to repeat this construction using the higher coordinates.
In each step $p\left(\beta\right)$ will be such a decision tree for
the first $\omega_{\beta}$ values of $\dot{F}$.

So assume now that we've already defined $p\left(\beta\right)$, and
we shall show how to define $p\left(\beta+1\right)$. We use $\mathbb{Q}_{\beta+1}$
to decide the values of $\dot{F}$ up to $\omega_{\beta+1}$. If $\beta$
is a limit ordinal we also set $A_{\beta}=\emptyset$.

Let $H_{\beta+1}\subsetneq\mathbb{P}_{\beta+1}$ be any generic set
such that $\left\langle p\left(\beta\right)^{0},\dot{p}\left(\beta\right)^{1},...,\dot{p}\left(\beta\right)^{\alpha},...\mid\alpha\leq\beta\right\rangle \in H_{\beta+1}$.

Now, working in $M_{0}\left[H_{\beta+1}\right]$ we repeat the construction.
To start the induction, set $p_{\omega_{\beta}}=p\left(\beta\right)^{>\beta}\in\nicefrac{\mathbb{P}_{\delta}}{H_{\beta+1}}$.

Assume $p_{\alpha}$ is defined, we are going to define $p_{\alpha+1}$.
Remember $\mathbb{Q}_{\beta+1}$ is $\aleph_{\beta+1}$-Sacks forcing.
So let $S_{\alpha}$ be the set of all order $\alpha$ splitting nodes
of $p_{\alpha}^{\beta+1}$. For each $s\in S_{\alpha}$ let $q_{s^{\frown}0},q_{s^{\frown}1}$
and $a_{s^{\frown}0},a_{s^{\frown}1}$ be such that $q_{s^{\frown}i}\geq\left\langle p_{\alpha}^{\beta+1}\upharpoonright s^{\frown}i,\dot{p}_{\alpha}^{>\beta+1}\right\rangle $
and $q_{s^{\frown}i}\Vdash\dot{F}(\alpha)=a_{s^{\frown}i}$. Let $p_{\alpha+1}^{\beta+1}=\bigcup\limits _{s\in S_{\alpha},i=0,1}q_{s^{\frown}i}^{\beta+1}$
be the amalgamation of the $q_{s^{\frown}i}$'s first coordinate.
As for the rest of the coordinates define $\dot{p}_{\alpha+1}^{>\beta+1}$
with accordance to the path taken in the first coordinate. Meaning
that $q_{s^{\frown}i}^{\beta+1}$ forces $\dot{p}_{\alpha+1}^{>\beta+1}=\dot{q}_{s^{\frown}i}^{>\beta+1}$.

From the way we defined the $q_{s^{\frown}i}$'s it is clear $\dot{p}_{\alpha+1}^{>\beta+1}\geq\dot{p}_{\alpha}^{>\beta+1}$,
and so $p_{\alpha+1}^{>\beta}\geq p_{\alpha}^{>\beta}$. Also, all
splitting nodes of order $\leq\alpha$ are preserved in $p_{\alpha+1}^{\beta+1}$,
and therefore $p_{\alpha+1}^{\beta+1}\geq_{\alpha}p_{\alpha}^{\beta+1}$.

In limit stages $\tau$ we simply define $p_{\tau}=\bigcap\limits _{\omega_{\beta}\leq\alpha<\tau}p_{\alpha}$.
In each coordinate of $p_{\tau}$ we have $\aleph_{\beta+1}$-closure
according to lemma \ref{lem:IterationIsClosed}. So for $\tau<\omega_{\beta+1}$
$p_{\tau}\in\nicefrac{\mathbb{P}_{\delta}}{H_{\beta+1}}$. For the
case $\tau=\omega_{\beta+1}$ note that we have a fusion sequence
in the first coordinate. So thanks to lemma \ref{lem:Fusion:-Let-}
and the $\aleph_{\beta+2}$-closure of the higher coordinates we have
$p_{\omega_{\beta+1}}\in\nicefrac{\mathbb{P}_{\delta}}{H_{\beta+1}}$.

Now, for each $a_{s^{\frown}i}$ we used, we pick a $\mathbb{P}_{\beta+1}$-name
$\dot{a}_{s^{\frown}i}$ in $M_{0}$. Because $\left|\mathbb{P}_{\beta+1}\right|<\delta$
we can find in $M_{0}$ a set $A_{\beta+1}$ of cardinality $<\delta$
such that $p_{\omega_{\beta}}\Vdash\check{A}_{\beta+1}\supseteq\left\{ \dot{a}_{s^{\frown}i}\mid s\in S_{\alpha}\wedge i=0,1\wedge\omega_{\beta}\leq\alpha<\omega_{\beta+1}\right\} $.

We also pick a name for $p_{\omega_{\beta+1}}$ such that 
\[
p\left(\beta\right)^{\leq\beta}\Vdash_{\mathbb{P}_{\beta+1}}\left\langle \dot{p}_{\omega_{\beta+1}}^{\beta+1},\dot{p}_{\omega_{\beta+1}}^{\beta+2},...,\dot{p}_{\omega_{\beta+1}}^{\eta},...\mid\eta<\delta\right\rangle \text{ is as we constructed it}.
\]

We set $p\left(\beta+1\right)=\left\langle p\left(0\right)^{0},\dot{p}\left(1\right)^{1},...,\dot{p}\left(\beta\right)^{\beta},\dot{p}_{\omega_{\beta+1}}^{\beta+1},...,\dot{p}_{\omega_{\beta+1}}^{\eta},...\mid\eta<\delta\right\rangle \in\mathbb{P}_{\delta}$.
Obviously $p\left(\beta+1\right)\geq p\left(\beta\right)$.

For limit stages $\tau$ we define $p\left(\tau\right)=\bigcap\limits _{\beta<\tau}p\left(\beta\right)$.
We claim $p\left(\tau\right)\in\mathbb{P}_{\delta}$. For each coordinate
$\beta<\tau$ note that the condition stabilizes, and so $p\left(\tau\right)^{\beta}=p\left(\beta\right)^{\beta}$.
For coordinates $\geq\tau$ lemma \ref{lem:IterationIsClosed} provides
at least $\aleph_{\tau+1}$-closure, and because we're using full
support this shows $p\left(\tau\right)\in\mathbb{P}_{\delta}$.

Therefore by induction we can construct condition $p\left(\delta\right)=\left\langle p\left(0\right)^{0},\dot{p}\left(1\right)^{1},...,\dot{p}\left(\eta\right)^{\eta},...\mid\eta<\delta\right\rangle \in\mathbb{P}_{\delta}$
.

Now define $A=\bigcup\limits _{\alpha<\delta}A_{\alpha}$. Because
for all $\alpha<\delta$ $OK\left(\alpha\right)$, we have $\left|A\right|=\delta$.

Observe that $p\left(\delta\right)\Vdash ran\left(\dot{F}\right)\subseteq\check{A}$.
Now if $\aleph_{\delta+1}$ is not preserved then in $M_{\delta}$
$\left|\aleph_{\delta+1}^{M_{0}}\right|=\delta$, and we can take
$F$ to be the bijection between $\aleph_{\delta+1}^{M_{0}}$ and
$\delta$. Applying the construction to this $F$ we get $p\left(\delta\right)\Vdash\aleph_{\delta+1}^{M_{0}}\subseteq\check{A}$.
But then we get in $M_{0}$ a surjection from $\delta$ onto $\aleph_{\delta+1}$,
which is impossible. Therefore $p\left(\delta\right)\Vdash\aleph_{\delta+1}^{M_{0}}=\aleph_{\delta+1}$.

We know that $p\left(\delta\right)\geq p\left(0\right)\geq p$, and
so by density we can assume without loss of generality that $p\left(\delta\right)\in G_{\delta}$.
Therefore $\aleph_{\delta+1}$ is indeed preserved.

Thus all cardinals are preserved in all cases, and overall.
\end{proof}
\begin{lem}
\label{lem:2^kap-Cond2}Let $\delta\leq\zeta$ be a limit ordinal.
If for all $\alpha<\delta$ $OK\left(\alpha\right)$, then $M_{\delta}\vDash2^{<\aleph_{\delta}}=\aleph_{\delta}$.
\end{lem}
\begin{proof}
Note that obviously $M_{\delta}\vDash2^{<\aleph_{\delta}}\geq\aleph_{\delta}$.

Now assume $M_{\delta}\vDash2^{<\aleph_{\delta}}>\aleph_{\delta}$.
That means for some $\alpha<\delta$ $M_{\delta}\vDash2^{\aleph_{\alpha}}>\aleph_{\delta}$.
However by lemma \ref{lem:IterationIsClosed} $\nicefrac{\mathbb{P}_{\delta}}{G_{\alpha+1}}$
is $\aleph_{\alpha+1}$-closed, and so new subsets of $\aleph_{\alpha}$
are added when forcing $M_{\delta}=M_{\alpha+1}\left[G_{\alpha+1,\delta}\right]$.
Therefore $M_{\alpha+1}\vDash2^{\aleph_{\alpha}}>\aleph_{\delta}$
and so $M_{\alpha+1}\vDash2^{\aleph_{\alpha+1}}>\aleph_{\alpha+2}$,
in contradiction to $OK\left(\alpha+1\right)$.

Therefore $M_{\delta}\vDash2^{<\aleph_{\delta}}=\aleph_{\delta}$.
\end{proof}
\begin{lem}
\label{lem:GCH-aboveIter}Let $\delta\leq\zeta$ be a limit ordinal.
If for all $\alpha<\delta$ $OK\left(\alpha\right)$, then $M_{\alpha}\vDash\forall\lambda\geq\aleph_{\alpha}\left(2^{\lambda}=\lambda^{+}\right)$.
\end{lem}
\begin{proof}
Recall the proof of lemma \ref{lem:GCHpres} in the previous section.
As already shown in lemma \ref{lem:SameCardinalsInLimit} $\left|\mathbb{P}_{\delta}\right|=\aleph_{\delta+1}$.
Therefore there are at most $\aleph_{\delta+1}^{\aleph_{\delta+1}}=\aleph_{\delta+2}$
different antichains. With the rest of the proof identical, we get
$M_{\delta}\vDash\forall\lambda>\aleph_{\delta}\left(2^{\lambda}=\lambda^{+}\right)$.

We now show that $M_{\delta}\vDash2^{\aleph_{\delta}}=\aleph_{\delta+1}$.

Let $A\subseteq\aleph_{\delta}$, and assume $p\in\mathbb{P}_{\delta}$
such that $p\vDash\dot{A}\subseteq\aleph_{\delta}\wedge\forall\beta<\delta\left(\dot{A}\notin M_{\beta}\right)$.
We denote individual coordinates like so: $p=\left\langle p^{0},\dot{p}^{1},...,\dot{p}^{\alpha},...\mid\alpha<\delta\right\rangle $.

Let $\dot{z}$ be the name of the characteristic function of $A$.
For any condition $q\in\mathbb{P}_{\delta}$ stronger than $p$ let
$\dot{z}_{q}$ be the longest initial segment of $\dot{z}$ that is
decided by $q$, let $\gamma_{q}$ be the first ordinal for which
$\dot{z}$ is undecided, let $q^{0}$ denote the first coordinate
of $q$, i.e from $\aleph_{0}$-Sacks forcing over $M_{0}$, and let
$\vec{q}$ denote the name of the rest of the coordinates. $\dot{z_{q}}$
and $\gamma_{q}$ must be well-defined, because if $q$ decides all
of $\dot{z}$, it decides all of $\dot{A}$, and then $A\in M_{0}$,
in contradiction to $p$. Plainly $\gamma_{q}<\omega_{\delta}$.

We now build a fusion sequence. Mark $p_{0}=p$. For the successor
case, assume that we've already chosen $p_{\epsilon}$. Note that
$p_{\epsilon}^{0}$ is just a perfect tree in $\mathbb{Q}_{0}$, and
$\vec{p_{\epsilon}}$ is a name in $M_{0}$. Let $S_{\epsilon}$ denote
the order $\epsilon$ splitting nodes of $p_{\epsilon}^{0}$. For
every splitting node $s\in S_{\epsilon}$, let's look at $\gamma_{\left\langle p_{\epsilon}^{0}\upharpoonright s,\vec{p_{\epsilon}}\right\rangle }$
and conditions $\left\langle p_{\epsilon}^{0}\upharpoonright s^{\frown}i,\vec{p_{\epsilon}}\right\rangle $.

Suppose that for both $i=0,1$ we have $\left\langle p_{\epsilon}^{0}\upharpoonright s^{\frown}i,\vec{p_{\epsilon}}\right\rangle \Vdash\dot{z}\left(\gamma_{\left\langle p_{\epsilon}^{0}\upharpoonright s,\vec{p_{\epsilon}}\right\rangle }\right)=j$.
$\gamma_{\left\langle p_{\epsilon}^{0}\upharpoonright s,\vec{p_{\epsilon}}\right\rangle }$
is undecided, so take $q\geq\left\langle p_{\epsilon}^{0}\upharpoonright s,\vec{p_{\epsilon}}\right\rangle $
such that $q\Vdash\dot{z}\left(\gamma_{\left\langle p_{\epsilon}^{0}\upharpoonright s,\vec{p_{\epsilon}}\right\rangle }\right)=1-j$.
Either $q\cap\left\langle p_{\epsilon}^{0}\upharpoonright s^{\frown}0,\vec{p_{\epsilon}}\right\rangle \in\mathbb{P}_{\delta}$
or $q\cap\left\langle p_{\epsilon}^{0}\upharpoonright s^{\frown}1,\vec{p_{\epsilon}}\right\rangle \in\mathbb{P}_{\delta}$.
But $q$ and $\left\langle p_{\epsilon}^{0}\upharpoonright s^{\frown}i,\vec{p_{\epsilon}}\right\rangle $
are incompatible for $i=0,1$, therefore our supposition is impossible.
Thus, if for a certain $i$ there is a $j$ such that $\left\langle p_{\epsilon}^{0}\upharpoonright s^{\frown}i,\vec{p_{\epsilon}}\right\rangle \Vdash\dot{z}\left(\gamma_{\left\langle p_{\epsilon}^{0}\upharpoonright s,\vec{p_{\epsilon}}\right\rangle }\right)=j$,
we can define $q_{s^{\frown}i}=\left\langle p_{\epsilon}^{0}\upharpoonright s^{\frown}i,\vec{p_{\epsilon}}\right\rangle $
so that $q_{s^{\frown}i}\Vdash\dot{z}\left(\gamma_{\left\langle p_{\epsilon}^{0}\upharpoonright s,\vec{p_{\epsilon}}\right\rangle }\right)=j$,
and know there will be some $q_{s^{\frown}\left(1-i\right)}\geq\left\langle p_{\epsilon}^{0}\upharpoonright s^{\frown}\left(1-i\right),\vec{p_{\epsilon}}\right\rangle $
such that $q_{s^{\frown}\left(1-i\right)}\Vdash\dot{z}\left(\gamma_{\left\langle p_{\epsilon}^{0}\upharpoonright s,\vec{p_{\epsilon}}\right\rangle }\right)=1-j$.

Alternatively, there is no such $i$, and we are free to select for
both $i=0,1$ $q_{s^{\frown}i}\geq\left\langle p_{\epsilon}^{0}\upharpoonright s^{\frown}i,\vec{p_{\epsilon}}\right\rangle $
such that $q_{s^{\frown}i}\Vdash\dot{z}\left(\gamma_{\left\langle p_{\epsilon}^{0}\upharpoonright s,\vec{p_{\epsilon}}\right\rangle }\right)=i$.

We now define the first coordinate as an amalgamation of the $q_{s^{\frown}i}^{0}$'s:
$p_{\epsilon+1}^{0}=\bigcup\limits _{s\in S_{\epsilon},i=0,1}q_{s^{\frown}i}^{0}$,
just like we did in the proof of theorem \ref{thm:OrgSacksMinimal}
of the original Sacks forcing. For the rest of the coordinates, we
define $\vec{p_{\epsilon+1}}$ with accordance to the path taken in
the first coordinate. Meaning that $q_{s^{\frown}i}^{0}$ forces $\vec{p_{\epsilon+1}}=\vec{q_{s^{\frown}i}}$.
From the way we defined the $q_{s^{\frown}i}$'s it is clear $\vec{p_{\epsilon+1}}\geq\vec{p_{\epsilon}}$,
and so $p_{\epsilon+1}\geq p_{\epsilon}$.

In the first coordinate we get a classical fusion sequence $\left\langle p_{\epsilon}^{0}\mid\epsilon<\omega\right\rangle $.
Therefore we can define $p_{\omega}^{0}=\bigcap\limits _{\epsilon<\omega}p_{\epsilon}^{0}\in\mathbb{Q}_{0}$.
As for the higher coordinates, we can define $\vec{p}_{\omega}=\bigcap\limits _{\epsilon<\omega}\vec{p}_{\epsilon}^{0}$
because of the $\aleph_{1}$-closure, proven in lemma \ref{lem:IterationIsClosed}.
So $p_{\omega}\in\mathbb{P}_{\delta}$.

So we can now define $p\left(0\right)=p_{\omega}$. Next we are going
to repeat by induction the construction above, using the higher coordinates.

So for the successor case, assume that we've already defined $p\left(\beta\right)$.
For all coordinates $0<\alpha\leq\beta$ define $\dot{p}\left(\beta+1\right)^{\alpha}=\dot{p}\left(\alpha\right)^{\alpha}$,
and let $p\left(\beta+1\right)^{0}=p\left(0\right)^{0}$.

Let $H_{\beta+1}\subseteq\mathbb{P}_{\beta+1}$ be any generic set
such that $\left\langle p\left(\beta\right)^{0},\dot{p}\left(\beta\right)^{1},...,\dot{p}\left(\beta\right)^{\alpha},...\mid\alpha\leq\beta\right\rangle \in H_{\beta+1}$.
Note that $M_{0}\left[H_{\beta+1}\right]=\text{the interpretation of }\dot{M}_{\beta+1}\text{ according to }\dot{H}_{\beta+1}$.

Working in $M_{0}\left[H_{\beta+1}\right]$, we know that $\left\langle p\left(\beta\right)^{\beta+1},\dot{p}\left(\beta\right)^{\beta+2},...,p\left(\beta\right)^{\alpha},...\mid\alpha<\delta\right\rangle \Vdash\dot{A}\notin M_{0}\left[H_{\beta+1}\right]$.
Just as before, let $\dot{z}$ be the name of the characteristic function
of $A$. For any condition $q\in\nicefrac{\mathbb{P}_{\delta}}{H_{\beta+1}}$
let $\dot{z}_{q}$ be the longest initial segment of $\dot{z}$ that
is decided by $q$, let $\gamma_{q}$ be the first ordinal for which
$\dot{z}$ is undecided, let $q^{\beta+1}$ denote the \emph{first}
coordinate of $q$, i.e from $\aleph_{\beta+1}$-Sacks forcing over
$M_{0}\left[H_{\beta+1}\right]$, and let $\vec{q}$ denote the name
of the rest of the coordinates. $\dot{z_{q}}$ and $\gamma_{q}$ must
be well-defined, because if $q$ decides all of $\dot{z}$, it decides
all of $\dot{A}$, and then $A\in M_{0}\left[H_{\beta+1}\right]$,
in contradiction to $p\left(\beta\right)$. Plainly $\gamma_{q}<\omega_{\beta+1}$.

Just as before, we again build a fusion sequence. This time mark $p_{0}=p\left(\beta\right)$.
For the successor case, assume that we've already chosen $p_{\epsilon}$.
Note that $p_{\epsilon}^{\beta+1}$ is just a perfect tree in $\mathbb{Q}_{\beta+1}$,
and $\vec{p_{\epsilon}}$ is a name of a condition. Let $S_{\epsilon}$
denote the order $\epsilon$ splitting nodes of $p_{\epsilon}^{\beta+1}$.
For every splitting node $s\in S_{\epsilon}$, let's look at $\gamma_{\left\langle p_{\epsilon}^{\beta+1}\upharpoonright s,\vec{p_{\epsilon}}\right\rangle }$
and conditions $\left\langle p_{\epsilon}^{\beta+1}\upharpoonright s^{\frown}i,\vec{p_{\epsilon}}\right\rangle $.

Suppose that for both $i=0,1$ we have $\left\langle p_{\epsilon}^{\beta+1}\upharpoonright s^{\frown}i,\vec{p_{\epsilon}}\right\rangle \Vdash\dot{z}\left(\gamma_{\left\langle p_{\epsilon}^{\beta+1}\upharpoonright s,\vec{p_{\epsilon}}\right\rangle }\right)=j$.
$\gamma_{\left\langle p_{\epsilon}^{\beta+1}\upharpoonright s,\vec{p_{\epsilon}}\right\rangle }$
is undecided, so take $q\geq\left\langle p_{\epsilon}^{\beta+1}\upharpoonright s,\vec{p_{\epsilon}}\right\rangle $
such that $q\Vdash\dot{z}\left(\gamma_{\left\langle p_{\epsilon}^{\beta+1}\upharpoonright s,\vec{p_{\epsilon}}\right\rangle }\right)=1-j$.
Either $q\cap\left\langle p_{\epsilon}^{\beta+1}\upharpoonright s^{\frown}0,\vec{p_{\epsilon}}\right\rangle \in\nicefrac{\mathbb{P}_{\delta}}{H_{\beta+1}}$
or $q\cap\left\langle p_{\epsilon}^{\beta+1}\upharpoonright s^{\frown}1,\vec{p_{\epsilon}}\right\rangle \in\nicefrac{\mathbb{P}_{\delta}}{H_{\beta+1}}$.
But $q$ and $\left\langle p_{\epsilon}^{\beta+1}\upharpoonright s^{\frown}i,\vec{p_{\epsilon}}\right\rangle $
are incompatible for $i=0,1$, therefore our supposition is impossible.
Thus, if for a certain $i$ there is a $j$ such that $\left\langle p_{\epsilon}^{\beta+1}\upharpoonright s^{\frown}i,\vec{p_{\epsilon}}\right\rangle \Vdash\dot{z}\left(\gamma_{\left\langle p_{\epsilon}^{\beta+1}\upharpoonright s,\vec{p_{\epsilon}}\right\rangle }\right)=j$,
we can define $q_{s^{\frown}i}=\left\langle p_{\epsilon}^{\beta+1}\upharpoonright s^{\frown}i,\vec{p_{\epsilon}}\right\rangle $
so that $q_{s^{\frown}i}\Vdash\dot{z}\left(\gamma_{\left\langle p_{\epsilon}^{\beta+1}\upharpoonright s,\vec{p_{\epsilon}}\right\rangle }\right)=j$,
and know there will be some $q_{s^{\frown}\left(1-i\right)}\geq\left\langle p_{\epsilon}^{\beta+1}\upharpoonright s^{\frown}\left(1-i\right),\vec{p_{\epsilon}}\right\rangle $
such that $q_{s^{\frown}\left(1-i\right)}\Vdash\dot{z}\left(\gamma_{\left\langle p_{\epsilon}^{\beta+1}\upharpoonright s,\vec{p_{\epsilon}}\right\rangle }\right)=1-j$.

Alternatively, there is no such $i$, and we are free to select for
both $i=0,1$ $q_{s^{\frown}i}\geq\left\langle p_{\epsilon}^{\beta+1}\upharpoonright s^{\frown}i,\vec{p_{\epsilon}}\right\rangle $
such that $q_{s^{\frown}i}\Vdash\dot{z}\left(\gamma_{\left\langle p_{\epsilon}^{\beta+1}\upharpoonright s,\vec{p_{\epsilon}}\right\rangle }\right)=i$.

We now define the first coordinate as an amalgamation of the $q_{s^{\frown}i}^{\beta+1}$'s:
$p_{\epsilon+1}^{\beta+1}=\bigcup\limits _{s\in S_{\epsilon},i=0,1}q_{s^{\frown}i}^{\beta+1}$.
For the rest of the coordinates, we define $\vec{p_{\epsilon+1}}$
with accordance to the path taken in the first coordinate. Meaning
that $q_{s^{\frown}i}^{\beta+1}$ forces that $\vec{p_{\epsilon+1}}=\vec{q_{s^{\frown}i}}$.
From the way we defined the $q_{s^{\frown}i}$'s it is clear $\vec{p_{\epsilon+1}}\geq\vec{p_{\epsilon}}$,
and so $p_{\epsilon+1}\geq p_{\epsilon}$. In limit stages $\tau<\omega_{\beta+1}$
we just use the $\aleph_{\beta+1}$-closure to define $p_{\tau}=\bigcap\limits _{\alpha<\tau}p_{\alpha}$.

Now, in order to define $p_{\omega_{\beta+1}}$, note that in the
first coordinate we again get a fusion sequence $\left\langle p_{\epsilon}^{\beta+1}\mid\epsilon<\omega\right\rangle $.
Therefore $p_{\omega_{\beta+1}}^{\beta+1}=\bigcap\limits _{\epsilon<\omega_{\beta+1}}p_{\epsilon}^{\beta+1}\in\mathbb{Q}_{\beta+1}$.
As for the rest of the coordinates, we can use the $\aleph_{\beta+1}$-closure
proven in lemma \ref{lem:IterationIsClosed} to define $\vec{p_{\omega_{\beta+1}}}=\bigcap\limits _{\epsilon<\omega_{\beta+1}}\vec{p}_{\epsilon}$.

For all $\alpha>\beta$ we define $\dot{p}\left(\beta+1\right)^{\alpha}=p_{\omega_{\beta+1}}^{\alpha}$.

We now pick a name $\left\langle \dot{p}\left(\beta+1\right)^{\beta+1},\dot{p}\left(\beta+1\right)^{\beta+2},...,\dot{p}\left(\beta+1\right)^{\alpha},...\mid\alpha<\delta\right\rangle $
for the $\left\langle p\left(\beta+1\right)^{\beta+1},\dot{p}\left(\beta+1\right)^{\beta+2},...,\dot{p}\left(\beta+1\right)^{\alpha},...\mid\alpha<\delta\right\rangle $
that we constructed in $\nicefrac{\mathbb{P}_{\delta}}{H_{\beta+1}}$,
such that $\left\langle p\left(\beta\right)^{0},\dot{p}\left(\beta\right)^{1},...,\dot{p}\left(\beta\right)^{\beta}\right\rangle $
forces it to be the way it was defined.

Finally, set $p\left(\beta+1\right)=\left\langle p\left(0\right)^{0},\dot{p}\left(1\right)^{1},...,\dot{p}\left(\beta+1\right)^{\beta+1},\dot{p}\left(\beta+1\right)^{\beta+2},...,\dot{p}\left(\beta+1\right)^{\alpha}\mid\alpha<\delta\right\rangle $.

For limit stages $\delta$ we define $p\left(\delta\right)=\bigcap\limits _{\alpha<\delta}p\left(\alpha\right)$.
As each coordinate lesser than $\delta$ stabilizes, and by lemma
\ref{lem:IterationIsClosed} each coordinate $\geq\delta$ is at least
$\aleph_{\delta+1}$-closed, and because we're using full support,
$p\left(\delta\right)\in\mathbb{P}_{\delta}$.

As we've constructed $p\left(\delta\right)$ stronger than a general
$p\in\mathbb{P}_{\delta}$, then by density arguments, we may assume
without loss of generality $p\left(\delta\right)\in G_{\delta}$.

Just as in corollary \ref{cor:.2^<k}, we can view $p\left(\delta\right)$
as a mapping: $p\left(\delta\right)$ takes as input a sequence of
branches $\left\langle \dot{h}_{\alpha}\mid a<\delta\right\rangle $,
where each $\dot{h}_{\alpha}$ is a branch of $\dot{p}\left(\delta\right)^{\alpha}$,
and interprets $A$.

For each $G_{\alpha,\alpha+1}$, let $g_{\alpha,\alpha+1}$ be the
generic branch interdefinable with $G_{\alpha,\alpha+1}$. Note that
because of the way $p\left(\delta\right)$ was defined, $M_{\delta}\vDash p\left(\delta\right)\left(\left\langle g_{\alpha,\alpha+1}\mid\alpha<\delta\right\rangle \right)=A$.
Therefore, for a given $G_{\delta}$, there can't be two different
subsets of $\aleph_{\delta}$ that produce the same $p\left(\delta\right)\in\mathbb{P}_{\delta}$
in the construction above.

As $\left|\mathbb{P}_{\delta}\right|=\aleph_{\delta+1}$, there are
at most $\aleph_{\delta+1}$ new subsets of $\aleph_{\delta}$. Therefore,
$\left(2^{\aleph_{\delta}}=\aleph_{\delta+1}\right)^{M_{\delta}}$.
\end{proof}
\begin{lem}
\label{lem:cardinals_iteration}For all $\beta\leq\zeta$, $M_{\beta}$
has the same cardinals as $M_{0}$.
\end{lem}
\begin{proof}
By induction we prove $OK\left(\beta\right)$. By lemma \ref{lem:OK0}
$OK\left(0\right)$. The successor case is handled in lemma \ref{lem:OK-succ}.

For the limit case, lemmas \ref{lem:SameCardinalsInLimit}, \ref{lem:2^kap-Cond2}
and \ref{lem:GCH-aboveIter} show that if $OK\left(\alpha\right)$
is true for all $\alpha<\delta$, then $OK\left(\delta\right)$.

Therefore by induction $OK\left(\beta\right)$, and so $M_{\beta}$
has the same cardinals as $M_{0}$.
\end{proof}
Next, we want to verify that during the iteration we don't create
any inner model of $M_{\zeta}$ other than the $M_{\alpha}$'s for
$\alpha<\zeta$.

Note that while theorem \ref{thm:-Sacks-forcing-produces} shows that
applying $\aleph_{\alpha}$-Sacks forcing doesn't add any inner model
between $M_{\alpha}$ and $M_{\alpha+1}$, it says nothing about limit
stages. If $M_{\delta}$ is the limit model, then theoretically there
might be another inner model lurking \emph{between} $M_{\delta}$
and $\bigcup\limits _{\alpha<\delta}M_{\alpha}$.

A second type of problem could arise even in the successor stages.
Applying the forcing over $M_{\alpha}$ with $\alpha>0$, one might
inadvertantly create some new inner model between $M_{0}=L$ and $M_{\alpha+1}$
\emph{outside} the chain. Therefore we need to prove our construction
avoids creating both types of 'accidental' models.
\begin{lem}
\label{lem:no-max}If $N$ is an inner model of $M_{\zeta}$ such
that for all $\alpha\leq\zeta$ $N\neq M_{\alpha}$, and $\beta$
is the least ordinal such that $N\nsupseteq M_{\beta}$, then $\beta$
is a limit ordinal.
\end{lem}
\begin{proof}
It is enough to show that there is no greatest $\beta$ such that
$M_{\beta}\subsetneq N$. Working to the contrary, assume $\beta<\zeta$
is such that $M_{\beta}\subsetneq N$ but $M_{\beta+1}\nsubseteq N$.

Obviously if $\mathbb{Q}_{\beta}$ is trivial then $M_{\beta+1}\subsetneq N$
in contradiction to the assumption. Therefore we may assume $\mathbb{Q}_{\beta}$
is $\aleph_{\beta}$-Sacks forcing.

$N$ is a model of $\mathsf{ZFC}$ between $M_{\beta}$ and $M_{\zeta}$.
Therefore according to lemma 15.43 in Jech \citep{key-17}, $N=M_{\beta}\left[A\right]$
for some set of ordinals $A\in M_{\zeta}$.

$N\neq M_{\beta}$ so obviously $A\in N\setminus M_{\beta}$. Let
$\dot{A}$ be its name in $M_{\beta}$ in the forcing $\nicefrac{\mathbb{P}_{\zeta}}{G_{\beta}}$.
There is an ordinal $\nu$ such that $0\Vdash\dot{A}\subseteq\nu$,
and let $\dot{z}$ be the name of the characteristic function of $A$,
$z\colon\nu\rightarrow2$. $A\notin M_{\beta}$, so there is a condition
$p\in G_{\beta,\zeta}$ that forces $\dot{A}\notin M_{\beta}$.

For any condition $q\in\nicefrac{\mathbb{P}_{\zeta}}{G_{\beta}}$
stronger than $p$, let $\dot{z}_{q}$ be the longest initial segment
of $\dot{z}$ that is decided by $q$, let $\gamma_{q}$ be the first
ordinal for which $\dot{z}$ is undecided, let $q^{0}$ denote the
first coordinate of $q$, i.e from $\aleph_{\beta}$-Sacks forcing
over $M_{\beta}$, and let $\vec{q}$ denote the name of the rest
of the coordinates. $\dot{z_{q}}$ and $\gamma_{q}$ must be well-defined,
because if $q$ decides all of $\dot{z}$, it decides all of $\dot{A}$,
and then $A\in M_{\beta}$, in contradiction to $p$. Plainly $\gamma_{q}<\nu$.

We're now going to build a fusion sequence. Mark $p_{0}=p$. For the
successor case, assume that we've already chosen $p_{\epsilon}$.
Note that $p_{\epsilon}^{0}$ is just a perfect tree in $\mathbb{Q}_{\beta}$,
and $\vec{p_{\epsilon}}$ is a name in $M_{\beta}$. Let $S_{\epsilon}$
denote the order $\epsilon$ splitting nodes of $p_{\epsilon}^{0}$.
For every splitting node $s\in S_{\epsilon}$, let's look at $\gamma_{\left\langle p_{\epsilon}^{0}\upharpoonright s,\vec{p_{\epsilon}}\right\rangle }$
and conditions $\left\langle p_{\epsilon}^{0}\upharpoonright s^{\frown}i,\vec{p_{\epsilon}}\right\rangle $.

Suppose that for both $i=0,1$ we have $\left\langle p_{\epsilon}^{0}\upharpoonright s^{\frown}i,\vec{p_{\epsilon}}\right\rangle \Vdash\dot{z}\left(\gamma_{\left\langle p_{\epsilon}^{0}\upharpoonright s,\vec{p_{\epsilon}}\right\rangle }\right)=j$.
$\gamma_{\left\langle p_{\epsilon}^{0}\upharpoonright s,\vec{p_{\epsilon}}\right\rangle }$
is undecided, so take $q\geq\left\langle p_{\epsilon}^{0}\upharpoonright s,\vec{p_{\epsilon}}\right\rangle $
such that $q\Vdash\dot{z}\left(\gamma_{\left\langle p_{\epsilon}^{0}\upharpoonright s,\vec{p_{\epsilon}}\right\rangle }\right)=1-j$.
Either $q\cap\left\langle p_{\epsilon}^{0}\upharpoonright s^{\frown}0,\vec{p_{\epsilon}}\right\rangle \in\nicefrac{\mathbb{P}_{\zeta}}{G_{\beta}}$
or $q\cap\left\langle p_{\epsilon}^{0}\upharpoonright s^{\frown}1,\vec{p_{\epsilon}}\right\rangle \in\nicefrac{\mathbb{P}_{\zeta}}{G_{\beta}}$.
But $q$ and $\left\langle p_{\epsilon}^{0}\upharpoonright s^{\frown}i,\vec{p_{\epsilon}}\right\rangle $
are incompatible for $i=0,1$, therefore our supposition is impossible.
Thus, if for a certain $i$ there is a $j$ such that $\left\langle p_{\epsilon}^{0}\upharpoonright s^{\frown}i,\vec{p_{\epsilon}}\right\rangle \Vdash\dot{z}\left(\gamma_{\left\langle p_{\epsilon}^{0}\upharpoonright s,\vec{p_{\epsilon}}\right\rangle }\right)=j$,
we can define $q_{s^{\frown}i}=\left\langle p_{\epsilon}^{0}\upharpoonright s^{\frown}i,\vec{p_{\epsilon}}\right\rangle $
so that $q_{s^{\frown}i}\Vdash\dot{z}\left(\gamma_{\left\langle p_{\epsilon}^{0}\upharpoonright s,\vec{p_{\epsilon}}\right\rangle }\right)=j$,
and know there will be some $q_{s^{\frown}\left(1-i\right)}\geq\left\langle p_{\epsilon}^{0}\upharpoonright s^{\frown}\left(1-i\right),\vec{p_{\epsilon}}\right\rangle $
such that $q_{s^{\frown}\left(1-i\right)}\Vdash\dot{z}\left(\gamma_{\left\langle p_{\epsilon}^{0}\upharpoonright s,\vec{p_{\epsilon}}\right\rangle }\right)=1-j$.

Alternatively, there is no such $i$, and we are free to select for
both $i=0,1$ $q_{s^{\frown}i}\geq\left\langle p_{\epsilon}^{0}\upharpoonright s^{\frown}i,\vec{p_{\epsilon}}\right\rangle $
such that $q_{s^{\frown}i}\Vdash\dot{z}\left(\gamma_{\left\langle p_{\epsilon}^{0}\upharpoonright s,\vec{p_{\epsilon}}\right\rangle }\right)=i$.

We now define the first coordinate as an amalgamation of the $q_{s^{\frown}i}^{0}$'s:
$p_{\epsilon+1}^{0}=\bigcup\limits _{s\in S_{\epsilon},i=0,1}q_{s^{\frown}i}^{0}$,
just like we did in the proof of theorem \ref{thm:OrgSacksMinimal}
of the original Sacks forcing. For the rest of the coordinates, we
define $\vec{p_{\epsilon+1}}$ with accordance to the path taken in
the first coordinate. Meaning that if $q_{s^{\frown}i}^{0}\in G_{\beta,\beta+1}$
then $\vec{p_{\epsilon+1}}=\vec{q_{s^{\frown}i}}$. From the way we
defined the $q_{s^{\frown}i}$'s it is clear $\vec{p_{\epsilon+1}}\geq\vec{p_{\epsilon}}$,
and so $p_{\epsilon+1}\geq p_{\epsilon}$.

In limit stages $\delta\leq\omega_{\beta}$ we take $p_{\delta}=\bigcap\limits _{\epsilon<\delta}p_{\epsilon}$,
so in the first coordinate we get a fusion sequence $\left\langle p_{\epsilon}^{0}\mid\epsilon<\delta\right\rangle $
just like in the proof of theorem \ref{thm:-Sacks-forcing-produces}.
Therefore $p_{\delta}^{0}\in\mathbb{Q}_{\beta}$. As for the rest
of the coordinates, $\vec{p_{\delta}}\in\dot{\nicefrac{\mathbb{P}_{\zeta}}{G_{\beta+1}}}$
because of the $\aleph_{\beta+1}$-closure, as shown in lemma \ref{lem:IterationIsClosed}.

So we can now define $q=p_{\omega_{\beta}}\in\nicefrac{\mathbb{P}_{\zeta}}{G_{\beta}}$.
Note that we constructed such a $q\geq p$ over \emph{any} $p\in G_{\beta,\zeta}$,
so due to density we may assume without loss of generality that $q\in G_{\beta,\zeta}$.

Now let $f=\left\{ s\in q^{0}\mid\dot{z}_{\left\langle q^{0}\upharpoonright s,\vec{q}\right\rangle }\subseteq\dot{z}_{G_{\beta,\zeta}}\right\} $.
We claim $f$ is a branch of $q^{0}$.

From density we know that for every $\alpha<\omega_{\beta}$ there
is an $r\in G_{\beta,\zeta}$ such that $r^{0}$ has a stem with length
at least $\alpha$. Because $G_{\beta,\zeta}$ is generic, there is
some condition $t\geq q\cap r$ in $G_{\beta,\zeta}$. This $t^{0}$
has a stem with length at least $\alpha$, and so there is some node
$s$ in level $\alpha$ of $q^{0}$ such that $t^{0}\geq q^{0}\upharpoonright s$.
Obviously $\dot{z}_{\left\langle q^{0}\upharpoonright s,\vec{q}\right\rangle }\subseteq\dot{z}_{\left\langle t^{0},\vec{q}\right\rangle }\subseteq\dot{z}_{G_{\beta,\zeta}}$,
and so for every $\alpha<\omega_{\beta}$ there is some $s$ in that
level of $q^{0}$ such that $s\in f$. Also, if $s\in f$, then it's
trivial that for all $\alpha$ $s\upharpoonright\alpha\in f$.

Next, we show that $f$ has no splitting nodes. Suppose $s$ is a
splitting node of $q$, then for $i=0,1$ $\dot{z}_{\left\langle q^{0}\upharpoonright s^{\frown}0,\vec{q}\right\rangle }\left(\gamma_{\left\langle q^{0}\upharpoonright s,\vec{q}\right\rangle }\right)\neq\dot{z}_{\left\langle q^{0}\upharpoonright s^{\frown}1,\vec{q}\right\rangle }\left(\gamma_{\left\langle q^{0}\upharpoonright s,\vec{q}\right\rangle }\right)$
and therefore either $\dot{z}_{G_{\beta,\zeta}}=\dot{z}_{\left\langle q^{0}\upharpoonright s^{\frown}0,\vec{q}\right\rangle }\left(\gamma_{\left\langle q^{0}\upharpoonright s,\vec{q}\right\rangle }\right)$
or $\dot{z}_{G_{\beta,\zeta}}=\dot{z}_{\left\langle q^{0}\upharpoonright s^{\frown}1,\vec{q}\right\rangle }\left(\gamma_{\left\langle q^{0}\upharpoonright s,\vec{q}\right\rangle }\right)$,
so either $s^{\frown}0$ or $s^{\frown}1$ is in $f$, but not both.
Therefore $s$ is not a splitting node in $f$, and so there are no
splitting nodes in $f$. We conclude that $f$ is indeed a branch
in $q^{0}$.

In fact, we claim that $f$ is equal to the generic branch $g_{\beta,\beta+1}$
derived from the generic set $G_{\beta,\zeta}$. Let $s\in g_{\beta,\beta+1}$.
Then due to density there is an $r\in G_{\beta,\zeta}$ such that
$s$ is part of the stem of $r^{0}$, and some condition $t\geq q\cap r$
in $G_{\beta,\zeta}$. As above, $t^{0}\geq q^{0}\upharpoonright s$,
and so $\dot{z}_{\left\langle q^{0}\upharpoonright s,\vec{q}\right\rangle }\subseteq\dot{z}_{\left\langle t^{0},\vec{q}\right\rangle }\subseteq\dot{z}_{G_{\beta,\zeta}}$.
Therefore $s\in f$.

Hence $f=g_{\beta,\beta+1}$ is a branch of $q^{0}$ that is definable
in $M_{\beta}\left[\dot{z}_{G_{\beta,\zeta}}\right]=M_{\beta}\left[\dot{A}_{G_{\beta,\zeta}}\right]$,
so $g_{\beta,\beta+1}\in M_{\beta}\left[\dot{A}_{G_{\beta,\zeta}}\right]$.
Meaning, using $\dot{A}_{G_{\beta,\zeta}}=A$, we managed to recover
the generic branch $g_{\beta,\beta+1}$. But remember, the generic
branch $g_{\beta,\beta+1}$ is in fact interdefinable with the generic
set $G_{\beta,\beta+1}$, and so $G_{\beta,\beta+1}\in M_{\beta}\left[A\right]$.
Therefore $M_{\beta+1}=M_{\beta}\left[G_{\beta,\beta+1}\right]=M_{\beta}\left[g_{\beta,\beta+1}\right]\subseteq M_{\beta}\left[A\right]\subseteq N$,
in contradiction to our assumption that $M_{\beta+1}\nsubseteq N$.

We conclude that if $N$ violates the theorem, there is no greatest
$\beta$ such that $M_{\beta}\subsetneq N$. Thus, the least inner
model of the tower that isn't included in $N$ must be $M_{\delta}$
for some limit ordinal $\delta$.
\end{proof}
\begin{lem}
\label{lem:lim_case}If $N$ is an inner model of $M_{\zeta}$, and
$\delta$ is a limit ordinal such that for all $\beta<\delta$ $N\supseteq M_{\beta}$,
then $N\supseteq M_{\delta}$.
\end{lem}
\begin{proof}
We show this inductively. So let $\delta$ be a limit ordinal, and
assume the lemma is true for every limit ordinal $\epsilon<\delta$.
Let $N$ be an inner model of $M_{\zeta}$ such that for all $\beta<\delta$
$M_{\beta}\subsetneq N$. We aim to show that $G_{\delta}\in N$ by
showing that $\left\langle G_{\beta,\beta+1}\mid\beta<\delta\right\rangle \in N$.

To start things off we first want to define a sequence $g=\left\langle g_{\beta}\mid\beta<\delta\right\rangle \in N$
such that for each $\beta<\delta$ if $\beta$ is not a limit ordinal
then $g_{\beta}\subsetneq\aleph_{\beta}$ and $g_{\beta}\notin M_{\beta}$.
Note that while each $M_{\beta}$ is by itself definable in $N$ using
set parameters, the sequence might not be, so we can't simply define
$A_{\beta}=\left\{ a\subsetneq\aleph_{\beta}\mid a\in N\setminus M_{\beta}\wedge\sup\left(a\right)=\aleph_{\beta}\right\} $
and then choose some $g_{\beta}\in A_{\beta}$ whenever $\beta$ is
not a limit.

Instead, we build this sequence inductively, working in $N$. Let
$N_{0}=L$. Next, for all $\beta<\delta$, assuming $N_{\beta}$ is
defined, let $A_{\beta}=\left\{ a\subsetneq\aleph_{\beta}\mid a\in N\setminus N_{\beta}\wedge\sup\left(a\right)=\aleph_{\beta}\right\} $,
and if $A_{\beta}\neq\emptyset$ choose some $g_{\beta}\in A_{\beta}$,
otherwise set $g_{\beta}=\emptyset$. For the successor step, assuming
$N_{\beta}$ is defined, we define $N_{\beta+1}=N_{\beta}\left[g_{\beta}\right]$.
In the limit step, assuming $N_{\beta}$ is defined for all $\beta<\epsilon<\delta$,
we define $N_{\epsilon}$ as the least inner model that includes every
$N_{\beta}$.

We claim that for all $\beta<\delta$ $N_{\beta}$ is definable and
equal to $M_{\beta}$, and that if $\beta$ is not a limit ordinal
then $g_{\beta}\subsetneq\aleph_{\beta}$ and $g_{\beta}\notin M_{\beta}$.
For the base case, note that $N_{0}=L=M_{0}$, which is of course
definable in $N$. Next, assuming $N_{\beta}=M_{\beta}$, then $A_{\beta}=\left\{ a\subsetneq\aleph_{\beta}\mid a\in N\setminus M_{\beta}\wedge\sup\left(a\right)=\aleph_{\beta}\right\} $.
If $\beta$ is a limit, then $\mathbb{Q}_{\beta}$ is trivial, and
so $M_{\beta}=M_{\beta+1}$. On the other hand, $\dot{\nicefrac{\mathbb{P}_{\zeta}}{G_{\beta}}}$
is $\aleph_{\beta+1}$-closed, as shown by lemma \ref{lem:IterationIsClosed}.
Therefore $M_{\beta}$ and $M_{\zeta}$ have the same subsets of $\aleph_{\beta}$.
Hence $A_{\beta}=\emptyset$, and therefore $g_{\beta}=\emptyset$.
We get $N_{\beta+1}=N_{\beta}\left[\emptyset\right]=N_{\beta}=M_{\beta}=M_{\beta+1}$
as required.

If $\beta$ is not a limit, then $\mathbb{Q}_{\beta}$ is $\aleph_{\beta}$-Sacks
forcing, and therefore there is a new subset of $\aleph_{\beta}$
in $M_{\beta+1}\setminus M_{\beta}$. Hence $A_{\beta}\neq\emptyset$.
On the other hand $\dot{\nicefrac{\mathbb{P}_{\zeta}}{G_{\beta+1}}}$
is $\aleph_{\beta+1}$-closed, so $M_{\beta+1}$ has the same subsets
of $\aleph_{\beta}$ as $M_{\zeta}$. Therefore $M_{\beta}\subsetneq M_{\beta}\left[g_{\beta}\right]\subseteq M_{\beta+1}$.
But $M_{\beta+1}$ is generated from $M_{\beta}$ using $\aleph_{\beta}$-Sacks
forcing, and so according to theorem \ref{thm:-Sacks-forcing-produces}
there is no intermediate model. Hence $N_{\beta+1}=N_{\beta}\left[g_{\beta}\right]=M_{\beta}\left[g_{\beta}\right]=M_{\beta+1}$.

In the limit step, assume that for $\epsilon$ a limit ordinal we've
already shown that $M_{\beta}=N_{\beta}$ for all $\beta<\epsilon$.
$G_{\epsilon}\in N$, and therefore $M_{\epsilon}$ is definable with
set parameters in $N$. Hence $N$ recognizes that $M_{\epsilon}$
is its inner model. Working towards a contradiction, assume $K$ is
an inner model of $N$ such that for all $\beta<\epsilon$ $M_{\beta}\subsetneq K$
but $M_{\epsilon}\nsubseteq K$. $K$ is definable with set parameters
in $N$, which is definable with set parameters in $M_{\zeta}$. Therefore
$K$ is definable with set parameters in $M_{\zeta}$, and therefore
$K$ is an inner model of $M_{\zeta}$ with said properties. But by
the induction hypothesis the lemma is true for every $\epsilon<\delta$,
so $K\supseteq M_{\epsilon}$ in contradiction to our assumption.
Therefore there is no such inner model $K$. So every inner model
of $N$ that includes all the $M_{\beta}$'s for $\beta<\epsilon$
must necessarily include $M_{\epsilon}$. Therefore $M_{\epsilon}$
is the least inner model that includes every $M_{\beta}$. But this
exactly coincides with our definition of $N_{\epsilon}$, and so $N_{\epsilon}=M_{\epsilon}$.

Thus the induction is now complete and we've managed to define $N_{\beta}$
and show that it is in fact equal to $M_{\beta}$ for all $\beta<\delta$.
We've also shown that if $\beta$ is not a limit then $A_{\beta}\neq\emptyset$
and so $g_{\beta}\subsetneq\aleph_{\beta}$ and $g_{\beta}\notin M_{\beta}$,
as required. Therefore the set $\left\langle g_{\beta}\mid\beta<\delta\right\rangle \in N$
is exactly the set which we set out to define.

The sequence $\left\langle g_{\beta}\mid\beta<\delta\right\rangle \in M_{\zeta}$.
However, by lemma \ref{lem:IterationIsClosed} $\nicefrac{\mathbb{P}_{\zeta}}{G_{\delta}}$
is $\aleph_{\delta+1}$-closed, and for all $\beta<\delta$ $g_{\beta}\in M_{\delta}$.
Therefore $\left\langle g_{\beta}\mid\beta<\delta\right\rangle \in M_{\delta}$.

Hence there exists a condition $u\in G_{\delta}$ that forces 
\[
\dot{g}=\left\langle \dot{g}_{\beta}\mid\beta<\delta\right\rangle \wedge\forall\beta<\delta\left(\beta\text{ is not a limit}\rightarrow\left(\dot{g}_{\beta}\subsetneq\aleph_{\beta}\wedge\dot{g}_{\beta}\in\dot{M}_{\beta+1}\setminus\dot{M}_{\beta}\right)\right)
\]

By the definition of the forcing, we may assume that for each $\beta<\delta$
$\dot{g}_{\beta}$ is a $\mathbb{P}_{\beta}\star\mathbb{Q}_{\beta}$-name
of $g_{\beta}$.

Assume now that we have a condition $p\in\mathbb{P}_{\delta}$ stronger
than $u$. We denote individual coordinates like so: $p=\left\langle p^{0},\dot{p}^{1},...,\dot{p}^{\alpha},...\mid\alpha<\delta\right\rangle $.

We know that $p\Vdash\dot{g}_{0}\subsetneq\aleph_{0}\wedge\dot{g}_{0}\notin M_{0}$.
Let $\dot{z}$ be the name of the characteristic function of $\dot{g}_{0}$.
For any condition $q\in\mathbb{P}_{\delta}$ let $\dot{z}_{q}$ be
the longest initial segment of $\dot{z}$ that is decided by $q$,
let $\gamma_{q}$ be the first ordinal for which $\dot{z}$ is undecided,
let $q^{0}$ denote the first coordinate of $q$, i.e from $\aleph_{0}$-Sacks
forcing over $M_{0}$, and let $\vec{q}$ denote the name of the rest
of the coordinates. $\dot{z_{q}}$ and $\gamma_{q}$ must be well-defined,
because if $q$ decides all of $\dot{z}$, it decides all of $\dot{g}_{0}$,
and then $g_{0}\in M_{0}$, in contradiction to $p$. Plainly $\gamma_{q}<\omega$.

We now build a fusion sequence. Mark $p_{0}=p$. For the successor
case, assume that we've already chosen $p_{\epsilon}$. Note that
$p_{\epsilon}^{0}$ is just a perfect tree in $\mathbb{Q}_{0}$, and
$\vec{p_{\epsilon}}$ is a name in $M_{0}$. Let $S_{\epsilon}$ denote
the order $\epsilon$ splitting nodes of $p_{\epsilon}^{0}$. For
every splitting node $s\in S_{\epsilon}$, let's look at $\gamma_{\left\langle p_{\epsilon}^{0}\upharpoonright s,\vec{p_{\epsilon}}\right\rangle }$
and conditions $\left\langle p_{\epsilon}^{0}\upharpoonright s^{\frown}i,\vec{p_{\epsilon}}\right\rangle $.

Suppose that for both $i=0,1$ we have $\left\langle p_{\epsilon}^{0}\upharpoonright s^{\frown}i,\vec{p_{\epsilon}}\right\rangle \Vdash\dot{z}\left(\gamma_{\left\langle p_{\epsilon}^{0}\upharpoonright s,\vec{p_{\epsilon}}\right\rangle }\right)=j$.
$\gamma_{\left\langle p_{\epsilon}^{0}\upharpoonright s,\vec{p_{\epsilon}}\right\rangle }$
is undecided, so take $q\geq\left\langle p_{\epsilon}^{0}\upharpoonright s,\vec{p_{\epsilon}}\right\rangle $
such that $q\Vdash\dot{z}\left(\gamma_{\left\langle p_{\epsilon}^{0}\upharpoonright s,\vec{p_{\epsilon}}\right\rangle }\right)=1-j$.
Either $q\cap\left\langle p_{\epsilon}^{0}\upharpoonright s^{\frown}0,\vec{p_{\epsilon}}\right\rangle \in\mathbb{P}_{\delta}$
or $q\cap\left\langle p_{\epsilon}^{0}\upharpoonright s^{\frown}1,\vec{p_{\epsilon}}\right\rangle \in\mathbb{P}_{\delta}$.
But $q$ and $\left\langle p_{\epsilon}^{0}\upharpoonright s^{\frown}i,\vec{p_{\epsilon}}\right\rangle $
are incompatible for $i=0,1$, therefore our supposition is impossible.
Thus, if for a certain $i$ there is a $j$ such that $\left\langle p_{\epsilon}^{0}\upharpoonright s^{\frown}i,\vec{p_{\epsilon}}\right\rangle \Vdash\dot{z}\left(\gamma_{\left\langle p_{\epsilon}^{0}\upharpoonright s,\vec{p_{\epsilon}}\right\rangle }\right)=j$,
we can define $q_{s^{\frown}i}=\left\langle p_{\epsilon}^{0}\upharpoonright s^{\frown}i,\vec{p_{\epsilon}}\right\rangle $
so that $q_{s^{\frown}i}\Vdash\dot{z}\left(\gamma_{\left\langle p_{\epsilon}^{0}\upharpoonright s,\vec{p_{\epsilon}}\right\rangle }\right)=j$,
and know there will be some $q_{s^{\frown}\left(1-i\right)}\geq\left\langle p_{\epsilon}^{0}\upharpoonright s^{\frown}\left(1-i\right),\vec{p_{\epsilon}}\right\rangle $
such that $q_{s^{\frown}\left(1-i\right)}\Vdash\dot{z}\left(\gamma_{\left\langle p_{\epsilon}^{0}\upharpoonright s,\vec{p_{\epsilon}}\right\rangle }\right)=1-j$.

Alternatively, there is no such $i$, and we are free to select for
both $i=0,1$ $q_{s^{\frown}i}\geq\left\langle p_{\epsilon}^{0}\upharpoonright s^{\frown}i,\vec{p_{\epsilon}}\right\rangle $
such that $q_{s^{\frown}i}\Vdash\dot{z}\left(\gamma_{\left\langle p_{\epsilon}^{0}\upharpoonright s,\vec{p_{\epsilon}}\right\rangle }\right)=i$.

We now define the first coordinate as an amalgamation of the $q_{s^{\frown}i}^{0}$'s:
$p_{\epsilon+1}^{0}=\bigcup\limits _{s\in S_{\epsilon},i=0,1}q_{s^{\frown}i}^{0}$,
just like we did in the proof of theorem \ref{thm:OrgSacksMinimal}
of the original Sacks forcing. For the rest of the coordinates, we
define $\vec{p_{\epsilon+1}}$ with accordance to the path taken in
the first coordinate. Meaning that $q_{s^{\frown}i}^{0}$ forces $\vec{p_{\epsilon+1}}=\vec{q_{s^{\frown}i}}$.
From the way we defined the $q_{s^{\frown}i}$'s it is clear $\vec{p_{\epsilon+1}}\geq\vec{p_{\epsilon}}$,
and so $p_{\epsilon+1}\geq p_{\epsilon}$.

In the first coordinate we get a classical fusion sequence $\left\langle p_{\epsilon}^{0}\mid\epsilon<\omega\right\rangle $.
Therefore we can define $p_{\omega}^{0}=\bigcap\limits _{\epsilon<\omega}p_{\epsilon}^{0}\in\mathbb{Q}_{0}$.
As for the higher coordinates, we can define $\vec{p}_{\omega}=\bigcap\limits _{\epsilon<\omega}\vec{p}_{\epsilon}^{0}$
because of the $\aleph_{1}$-closure, proven in lemma \ref{lem:IterationIsClosed}.
So $p_{\omega}\in\mathbb{P}_{\delta}$.

So we can now define $p\left(0\right)=p_{\omega}$. Next we are going
to repeat by induction the construction above, using the higher coordinates.

So for the successor case, assume that we've already defined $p\left(\beta\right)$.
For all coordinates $0<\alpha\leq\beta$ define $\dot{p}\left(\beta+1\right)^{\alpha}=\dot{p}\left(\alpha\right)^{\alpha}$,
and let $p\left(\beta+1\right)^{0}=p\left(0\right)^{0}$. We are now
going to deal with $\dot{g}_{\beta+1}$. So let $H_{\beta+1}\subseteq\mathbb{P}_{\beta+1}$
be any generic set such that $\left\langle p\left(\beta\right)^{0},\dot{p}\left(\beta\right)^{1},...,\dot{p}\left(\beta\right)^{\alpha},...\mid\alpha\leq\beta\right\rangle \in H_{\beta+1}$.

Note that $M_{0}\left[H_{\beta+1}\right]=\text{the interpretation of }\dot{M}_{\beta+1}\text{ according to }\dot{H}_{\beta+1}$.
So working in $M_{0}\left[H_{\beta+1}\right]$, we know that 
\begin{multline*}
\left\langle p\left(\beta\right)^{\beta+1},\dot{p}\left(\beta\right)^{\beta+2},...,p\left(\beta\right)^{\alpha},...\mid\alpha<\delta\right\rangle \Vdash\dot{g}_{\beta+1}\subsetneq\aleph_{\beta+1}\wedge\dot{g}_{\beta+1}\notin M_{0}\left[H_{\beta+1}\right]
\end{multline*}
Just as before, let $\dot{z}$ be the name of the characteristic function
of $\dot{g}_{\beta+1}$. For any condition $q\in\nicefrac{\mathbb{P}_{\delta}}{H_{\beta+1}}$
let $\dot{z}_{q}$ be the longest initial segment of $\dot{z}$ that
is decided by $q$, let $\gamma_{q}$ be the first ordinal for which
$\dot{z}$ is undecided, let $q^{\beta+1}$ denote the \emph{first}
coordinate of $q$, i.e from $\aleph_{\beta+1}$-Sacks forcing over
$M_{0}\left[H_{\beta+1}\right]$, and let $\vec{q}$ denote the name
of the rest of the coordinates. $\dot{z_{q}}$ and $\gamma_{q}$ must
be well-defined, because if $q$ decides all of $\dot{z}$, it decides
all of $\dot{g}_{\beta+1}$, and then $g_{\beta+1}\in M_{0}\left[H_{\beta+1}\right]$,
in contradiction to $p\left(\beta\right)$. Plainly $\gamma_{q}<\omega_{\beta+1}$.

Just as before, we again build a fusion sequence. This time mark $p_{0}=p\left(\beta\right)$.
For the successor case, assume that we've already chosen $p_{\epsilon}$.
Note that $p_{\epsilon}^{\beta+1}$ is just a perfect tree in $\mathbb{Q}_{\beta+1}$,
and $\vec{p_{\epsilon}}$ is a name of a condition. Let $S_{\epsilon}$
denote the order $\epsilon$ splitting nodes of $p_{\epsilon}^{\beta+1}$.
For every splitting node $s\in S_{\epsilon}$, let's look at $\gamma_{\left\langle p_{\epsilon}^{\beta+1}\upharpoonright s,\vec{p_{\epsilon}}\right\rangle }$
and conditions $\left\langle p_{\epsilon}^{\beta+1}\upharpoonright s^{\frown}i,\vec{p_{\epsilon}}\right\rangle $.

Suppose that for both $i=0,1$ we have $\left\langle p_{\epsilon}^{\beta+1}\upharpoonright s^{\frown}i,\vec{p_{\epsilon}}\right\rangle \Vdash\dot{z}\left(\gamma_{\left\langle p_{\epsilon}^{\beta+1}\upharpoonright s,\vec{p_{\epsilon}}\right\rangle }\right)=j$.
$\gamma_{\left\langle p_{\epsilon}^{\beta+1}\upharpoonright s,\vec{p_{\epsilon}}\right\rangle }$
is undecided, so take $q\geq\left\langle p_{\epsilon}^{\beta+1}\upharpoonright s,\vec{p_{\epsilon}}\right\rangle $
such that $q\Vdash\dot{z}\left(\gamma_{\left\langle p_{\epsilon}^{\beta+1}\upharpoonright s,\vec{p_{\epsilon}}\right\rangle }\right)=1-j$.
Either $q\cap\left\langle p_{\epsilon}^{\beta+1}\upharpoonright s^{\frown}0,\vec{p_{\epsilon}}\right\rangle \in\nicefrac{\mathbb{P}_{\delta}}{H_{\beta+1}}$
or $q\cap\left\langle p_{\epsilon}^{\beta+1}\upharpoonright s^{\frown}1,\vec{p_{\epsilon}}\right\rangle \in\nicefrac{\mathbb{P}_{\delta}}{H_{\beta+1}}$.
But $q$ and $\left\langle p_{\epsilon}^{\beta+1}\upharpoonright s^{\frown}i,\vec{p_{\epsilon}}\right\rangle $
are incompatible for $i=0,1$, therefore our supposition is impossible.
Thus, if for a certain $i$ there is a $j$ such that $\left\langle p_{\epsilon}^{\beta+1}\upharpoonright s^{\frown}i,\vec{p_{\epsilon}}\right\rangle \Vdash\dot{z}\left(\gamma_{\left\langle p_{\epsilon}^{\beta+1}\upharpoonright s,\vec{p_{\epsilon}}\right\rangle }\right)=j$,
we can define $q_{s^{\frown}i}=\left\langle p_{\epsilon}^{\beta+1}\upharpoonright s^{\frown}i,\vec{p_{\epsilon}}\right\rangle $
so that $q_{s^{\frown}i}\Vdash\dot{z}\left(\gamma_{\left\langle p_{\epsilon}^{\beta+1}\upharpoonright s,\vec{p_{\epsilon}}\right\rangle }\right)=j$,
and know there will be some $q_{s^{\frown}\left(1-i\right)}\geq\left\langle p_{\epsilon}^{\beta+1}\upharpoonright s^{\frown}\left(1-i\right),\vec{p_{\epsilon}}\right\rangle $
such that $q_{s^{\frown}\left(1-i\right)}\Vdash\dot{z}\left(\gamma_{\left\langle p_{\epsilon}^{\beta+1}\upharpoonright s,\vec{p_{\epsilon}}\right\rangle }\right)=1-j$.

Alternatively, there is no such $i$, and we are free to select for
both $i=0,1$ $q_{s^{\frown}i}\geq\left\langle p_{\epsilon}^{\beta+1}\upharpoonright s^{\frown}i,\vec{p_{\epsilon}}\right\rangle $
such that $q_{s^{\frown}i}\Vdash\dot{z}\left(\gamma_{\left\langle p_{\epsilon}^{\beta+1}\upharpoonright s,\vec{p_{\epsilon}}\right\rangle }\right)=i$.

We now define the first coordinate as an amalgamation of the $q_{s^{\frown}i}^{\beta+1}$'s:
$p_{\epsilon+1}^{\beta+1}=\bigcup\limits _{s\in S_{\epsilon},i=0,1}q_{s^{\frown}i}^{\beta+1}$.
For the rest of the coordinates, we define $\vec{p_{\epsilon+1}}$
with accordance to the path taken in the first coordinate. Meaning
that $q_{s^{\frown}i}^{\beta+1}$ forces that $\vec{p_{\epsilon+1}}=\vec{q_{s^{\frown}i}}$.
From the way we defined the $q_{s^{\frown}i}$'s it is clear $\vec{p_{\epsilon+1}}\geq\vec{p_{\epsilon}}$,
and so $p_{\epsilon+1}\geq p_{\epsilon}$. In limit stages $\tau<\omega_{\beta+1}$
we just use the $\aleph_{\beta+1}$-closure to define $p_{\tau}=\bigcap\limits _{\alpha<\tau}p_{\alpha}$.

Now, in order to define $p_{\omega_{\beta+1}}$, note that in the
first coordinate we again get a fusion sequence $\left\langle p_{\epsilon}^{\beta+1}\mid\epsilon<\omega\right\rangle $.
Therefore $p_{\omega_{\beta+1}}^{\beta+1}=\bigcap\limits _{\epsilon<\omega_{\beta+1}}p_{\epsilon}^{\beta+1}\in\mathbb{Q}_{\beta+1}$.
As for the rest of the coordinates, we can use the $\aleph_{\beta+1}$-closure
proven in lemma \ref{lem:IterationIsClosed} to define $\vec{p_{\omega_{\beta+1}}}=\bigcap\limits _{\epsilon<\omega_{\beta+1}}\vec{p}_{\epsilon}$.

For all $\alpha>\beta$ we define $\dot{p}\left(\beta+1\right)^{\alpha}=p_{\omega_{\beta+1}}^{\alpha}$.

We now pick a name $\left\langle \dot{p}\left(\beta+1\right)^{\beta+1},\dot{p}\left(\beta+1\right)^{\beta+2},...,\dot{p}\left(\beta+1\right)^{\alpha},...\mid\alpha<\delta\right\rangle $
for the $\left\langle p\left(\beta+1\right)^{\beta+1},\dot{p}\left(\beta+1\right)^{\beta+2},...,\dot{p}\left(\beta+1\right)^{\alpha},...\mid\alpha<\delta\right\rangle $
that we constructed in $\nicefrac{\mathbb{P}_{\delta}}{H_{\beta+1}}$,
such that $\left\langle p\left(\beta\right)^{0},\dot{p}\left(\beta\right)^{1},...,\dot{p}\left(\beta\right)^{\beta}\right\rangle $
forces it to be the way it was defined.

Finally, set $p\left(\beta+1\right)=\left\langle p\left(0\right)^{0},\dot{p}\left(1\right)^{1},...,\dot{p}\left(\beta+1\right)^{\beta+1},\dot{p}\left(\beta+1\right)^{\beta+2},...,\dot{p}\left(\beta+1\right)^{\alpha}\mid\alpha<\delta\right\rangle $.

For limit stages $\delta$ we define $p\left(\delta\right)=\bigcap\limits _{\alpha<\delta}p\left(\alpha\right)$.
As each coordinate lesser than $\delta$ stabilizes, and by lemma
\ref{lem:IterationIsClosed} each coordinate $\geq\delta$ is at least
$\aleph_{\delta+1}$-closed, and because we're using full support,
$p\left(\delta\right)\in\mathbb{P}_{\delta}$.

As we've constructed $p\left(\delta\right)$ stronger than a general
$p\in\mathbb{P}_{\delta}$, then by density arguments, we may assume
without loss of generality $p\left(\delta\right)\in G_{\delta}$.

Now we're going to use $p\left(\delta\right)$ and the sequence of
$\left\langle g_{\beta}\mid\beta<\delta\right\rangle $ to recover
$\left\langle G_{\beta,\beta+1}\mid\beta<\delta\right\rangle $.

By induction, assume that for some $\beta<\delta$ we already recovered
$\left\langle G_{\alpha,\alpha+1}\mid\alpha<\beta\right\rangle $,
and thus $M_{\beta}$. Therefore, working in $M_{\beta}$, let $\dot{z}$
be as before the name of the characteristic function of $\dot{g}_{\beta}$
and for each $r\in\nicefrac{\mathbb{P}_{\delta}}{G_{\beta}}$ let
$\dot{z}_{r}$ be the longest initial segment of $\dot{z}$ that is
decided by $r$.

Let $q=p\left(\delta\right)$, and define $f=\left\{ s\in q^{\beta}\mid\dot{z}_{\left\langle q^{\beta}\upharpoonright s,\vec{q}\right\rangle }\subseteq\dot{z}_{G_{\beta,\delta}}\right\} $.
We claim that $f$ is a branch of $q^{\beta}$.

From density we know that for every $\alpha<\omega_{\beta}$ there
is an $r\in G_{\beta,\delta}$ such that $r^{\beta}$ has a stem with
length at least $\alpha$. Because $G_{\beta,\delta}$ is generic,
there is some condition $t\geq q\cap r$ in $G_{\beta,\delta}$. This
$t^{\beta}$ has a stem with length at least $\alpha$, and so there
is some node $s$ in level $\alpha$ of $q^{\beta}$ such that $t^{\beta}\geq q^{\beta}\upharpoonright s$.
Obviously $\dot{z}_{\left\langle q^{\beta}\upharpoonright s,\vec{q}\right\rangle }\subseteq\dot{z}_{\left\langle t^{\beta},\vec{q}\right\rangle }\subseteq\dot{z}_{G_{\beta,\delta}}$,
and so for every $\alpha<\omega_{\beta}$ there is some $s$ in that
level of $q^{\beta}$ such that $s\in f$. Also, if $s\in f$, then
it's trivial that for all $\alpha$ $s\upharpoonright\alpha\in f$.

Next, we show that $f$ has no splitting nodes. Suppose that $s$
is a splitting node of $q^{\beta}$, then for $i=0,1$ $\dot{z}_{\left\langle q^{\beta}\upharpoonright s^{\frown}0,\vec{q}\right\rangle }\left(\gamma_{\left\langle q^{\beta}\upharpoonright s,\vec{q}\right\rangle }\right)\neq\dot{z}_{\left\langle q^{\beta}\upharpoonright s^{\frown}1,\vec{q}\right\rangle }\left(\gamma_{\left\langle q^{\beta}\upharpoonright s,\vec{q}\right\rangle }\right)$
and therefore either $\dot{z}_{G_{\beta,\delta}}=\dot{z}_{\left\langle q^{\beta}\upharpoonright s^{\frown}0,\vec{q}\right\rangle }\left(\gamma_{\left\langle q^{\beta}\upharpoonright s,\vec{q}\right\rangle }\right)$
or $\dot{z}_{G_{\beta,\delta}}=\dot{z}_{\left\langle q^{\beta}\upharpoonright s^{\frown}1,\vec{q}\right\rangle }\left(\gamma_{\left\langle q^{\beta}\upharpoonright s,\vec{q}\right\rangle }\right)$,
so either $s^{\frown}0$ or $s^{\frown}1$ is in $f$, but not both.
Therefore $s$ is not a splitting node in $f$, and so there are no
splitting nodes in $f$. We conclude that $f$ is in fact a branch
in $q^{\beta}$.

Moreover, we claim that $f$ is equal to the generic branch $h_{\beta,\beta+1}$
derived from the first coordinate of the generic set $G_{\beta,\delta}$.
Let $s\in h_{\beta,\beta+1}$. Then due to density there is an $r\in G_{\beta,\delta}$
such that $s$ is part of the stem of $r^{\beta}$, and some condition
$t\geq q\cap r$ in $G_{\beta,\delta}$. As above, $t^{\beta}\geq q^{\beta}\upharpoonright s$,
and so $\dot{z}_{\left\langle q^{\beta}\upharpoonright s,\vec{q}\right\rangle }\subseteq\dot{z}_{\left\langle t^{\beta},\vec{q}\right\rangle }\subseteq\dot{z}_{G_{\beta,\zeta}}$.
Therefore $s\in f$.

Hence $f=h_{\beta,\beta+1}$ is a branch of $q^{\beta}$ that is definable
in $M_{\beta}\left[\dot{z}_{G_{\beta,\delta}}\right]=M_{\beta}\left[\dot{g_{\beta}}_{G_{\beta,\delta}}\right]$,
and therefore in $M_{\beta}\left[\dot{g}_{G_{\beta,\delta}}\right]$.
Meaning we managed to recover the generic branch $h_{\beta,\beta+1}$.
But remember, the generic branch $h_{\beta,\beta+1}$ is in fact interdefinable
with the generic set $G_{\beta,\beta+1}$, and so $G_{\beta,\beta+1}\in M_{\beta}\left[\dot{g}_{G_{\beta,\delta}}\right]$.
But by our inductive assumption $\left\langle G_{\alpha,\alpha+1}\mid\alpha<\beta\right\rangle \in M_{0}\left[\dot{g}_{G_{\delta}}\right]$.
Therefore $G_{\beta,\beta+1}\in M_{0}\left[\dot{g}_{G_{\delta}}\right]$.

Completing the induction, $G_{\delta}\in M_{0}\left[\dot{g}_{G_{\delta}}\right]$,
and so $M_{\delta}\subseteq M_{0}\left[\dot{g}_{G_{\beta,\delta}}\right]$.

But $M_{0}\left[\dot{g}_{G_{\delta}}\right]\subseteq N$, and so we
conclude $M_{\delta}\subseteq N$ as required.
\end{proof}
Looking at the proof of lemma \ref{lem:lim_case}, it becomes evident
why we couldn't have used bounded support even in regular limits:
condition $p\left(\delta\right)$ that lies at the heart of the proof
satisfies $\dot{p}\left(\delta\right)^{\alpha}\neq0$ whenever $\mathbb{Q}_{\alpha}$
is non-trivial. Moreover, had we used bounded support, we would have
needed to construct a condition $p$ that is at once bounded, and
so has at most than $\aleph_{\beta}<\aleph_{\delta}$ splitting nodes,
yet is still somehow able to distinguish the value of the generic
branch in $\mathbb{Q}_{\beta+1}$, even though in the standard forcing
that task requires $\aleph_{\beta+1}$ splitting nodes.

At last we arrive at the central theorem for the model tower construction:
\begin{thm}
\label{thm:chainOfModels}$K$ is an inner model of $M_{\zeta}$ if
and only if for some $\alpha\leq\zeta$ $K=M_{\alpha}$.
\end{thm}
\begin{proof}
Suppose to the contrary that $K$ is an inner model of $M_{\zeta}$
such that $K\neq M_{\alpha}$ for all $\alpha\leq\zeta$. Then there
is a minimal ordinal $\beta$ such that $K\nsupseteq M_{\beta}$.
By lemma \ref{lem:no-max} $\beta$ must be a limit ordinal. So $K\supsetneq M_{\alpha}$
for all $\alpha<\beta$, but $K\nsupseteq M_{\beta}$. However, according
to lemma \ref{lem:lim_case} if $K\supsetneq M_{\alpha}$ for all
$\alpha<\beta$ then $K\supseteq M_{\beta}$. We arrived at a contradiction.
Meaning that there is no such inner model $K$.

Thus $K$ is an inner model of $M_{\zeta}$ if and only if $K=M_{\alpha}$
for some $\alpha<\zeta$.
\end{proof}
\begin{cor}
There exists a well-ordered model universe of arbitrary height.
\end{cor}
\begin{proof}
$M_{\zeta}$ is a well-ordered model universe of height $\zeta$.
\end{proof}

\section{Class forcing}

In the previous section we defined the iterated forcing notion for
sets, and we used it to construct a well-ordered model universe of
arbitrary height. Because that iteration could successfully go through
strongly inaccessible cardinals, we proved that the existence of well-ordered
model universes with ordinal height is in fact consistent with $\mathsf{ZFC}$.
We could simply take $V\left[G_{\kappa}\right]_{\kappa}$, where $\kappa$
is a strongly inaccessible cardinal, and $G_{\kappa}$ is the generic
set of $\mathbb{P}_{\kappa}$ as defined in \ref{def:iteration}.

Now however we want to iterate our model tower 'all the way' by the
use of class forcing. And to make formal use of class forcing, we
return in this section to the axiomatic framework of $\mathsf{BGC}$,
as expounded upon in the introduction. So for the rest of this section
we shall assume to be working within $\left\langle V,\mathcal{V},\in\right\rangle $
a model of $\mathsf{BGC}$, and we shall use forcing to extend a base
model of $\mathsf{BGC}$ to another model thereof.

A basic introduction of class forcing the reader may be found in Friedman
\citep{key-32}. For a more thorough presentation of class forcing
within the context of $\mathsf{BGC}$ the reader may refer to Reitz
(appendix A of \citep{key-2}).

Before going on, it is important to note the main difficulty with
class forcing, which is that unlike set forcing, the generic extension
of class forcing might actually fail to be a model of $\mathsf{BGC}$
(and its sets a model of $\mathsf{ZFC}$). Specifically, the Power
Set Axiom and the Axiom of Replacement might fail (theorem 91 in \citep{key-2}).
For $\mathsf{BGC}$ and $\mathsf{ZFC}$ to be satisfied, we will need
to prove that our forcing iteration is \emph{progressively closed},
as will be defined later.

We fix the base of our forcing iteration to be $\left\langle L,\mathcal{L},\in\right\rangle $,
where of course $L$ is the constructible universe, and $\mathcal{L}$
is the collection of classes definable therein (remember fact \ref{fact:ZF->BG}).
\begin{defn}
\label{def:class-iter}Let $\mathbb{P}\in\mathcal{L}$ be a partially
ordered class defined as follows:
\begin{enumerate}
\item Let $\mathbb{\dot{Q}}_{\alpha}$ be trivial if $\alpha$ is a limit
ordinal, and let it be the name of $\aleph_{\alpha}$-Sacks forcing
in $V^{\mathbb{P}_{\alpha}}$ otherwise.
\item $\mathbb{P}_{\alpha+1}=\mathbb{P}_{\alpha}\star\mathbb{\dot{Q}}_{\alpha}$.
\item At limit stages we use full support, i.e if $\delta$ is a limit ordinal
then $p\in\mathbb{P}_{\delta}\Leftrightarrow\forall\alpha<\delta\left(p\upharpoonright\alpha\in\mathbb{P}_{\alpha}\right)$.
\item $\mathbb{P}=\bigcup\limits _{\alpha\in\mathbf{Ord}}\mathbb{P}_{\alpha}$.
That is, every condition in $\mathbb{P}$ is bounded in its coordinates.
\end{enumerate}
\end{defn}
It should be noted that unlike in the ordinal limit stages, where
we use the indirect limit (i.e full support) all the way through,
in the class limit we employ the direct limit instead.

As explained in the previous section, using indirect limits even for
regular cardinals would have spoiled the construction of the condition
used to simultaneously discover all the generic sets - which was necessary
to prove that no inner model 'squeezes in' between the ascending chain
of models and the limit model. But as will be shown later, unlike
the ordinal limit stages, if we use a direct limit in the class stage
the generic extension is simply the union of the ascending chain,
and therefore automatically minimal over it. So the entire construction
of theorem \ref{thm:chainOfModels} is unnecessary for the class limit
case.
\begin{defn}
\label{def:class-notation}Denote:
\begin{enumerate}
\item $M_{0}=L$.
\item $G_{\alpha}$ as the generic set in partial order $\mathbb{P}_{\alpha}$
over $M_{0}$.
\item $M_{\alpha}=M_{0}\left[G_{\alpha}\right]$.
\item $\mathbb{G}$ as the generic class in partial order $\mathbb{P}$
over $M_{0}$.
\item $M_{0}\left[\mathbb{G}\right]$ the generic extension of $\left\langle L,\mathcal{L},\in\right\rangle $
by $\mathbb{G}$.
\item $M_{\infty}$ the restriction of $M_{0}\left[\mathbb{G}\right]$ to
sets.
\end{enumerate}
\end{defn}
Note that we have yet to establish that $M_{0}\left[\mathbb{G}\right]$
is a model of $\mathsf{BGC}$, or that $M_{\infty}$ is a model of
$\mathsf{ZFC}$.
\begin{lem}
\label{lem:ClassRelativeClosure}The forcing $\nicefrac{\mathbb{P}}{G_{\alpha}}$
is $\aleph_{\alpha}$-closed.
\end{lem}
\begin{proof}
Each coordinate is $\aleph_{\alpha}$-closed, and the limit of a set
of bounded conditions in $\nicefrac{\mathbb{P}}{G_{\alpha}}$ is itself
bounded.
\end{proof}
\begin{lem}
\label{lem:ClassSameCard}$M_{\infty}$ has the same cardinals as
$M_{0}$.
\end{lem}
\begin{proof}
Let $\aleph_{\alpha}$ be a cardinal in $M_{0}$. According to lemma
\ref{lem:cardinals_iteration} $M_{0}$ has the same cardinals as
$M_{\alpha+1}$.

But by lemma \ref{lem:ClassRelativeClosure} the forcing $\nicefrac{\mathbb{P}}{G_{\alpha+1}}$
is $\aleph_{\alpha+1}$-closed, and so adds no new subsets of $\aleph_{\alpha}$.
Therefore $\aleph_{\alpha}^{M_{0}}=\aleph_{\alpha}^{M_{\alpha+1}}=\aleph_{\alpha}^{M_{\infty}}$,
and so all cardinals are preserved.
\end{proof}
\begin{lem}
$M_{\infty}$ satisfies the Power Set Axiom.
\end{lem}
\begin{proof}
It is enough to prove the Power Set Axiom for cardinals. Let $\aleph_{\alpha}$
be a cardinal in $M_{\infty}$. By lemma \ref{lem:ClassSameCard}
it is also a cardinal in $M_{\alpha+1}$.

By lemma \ref{lem:ClassRelativeClosure} the forcing $\nicefrac{\mathbb{P}}{G_{\alpha+1}}$
is $\aleph_{\alpha+1}$-closed, and so adds no new subsets of $\aleph_{\alpha}$.
Therefore $\mathcal{P}\left(\aleph_{\alpha}\right)^{M_{\alpha+1}}=\mathcal{P}\left(\aleph_{\alpha}\right)^{M_{\infty}}$.

Thus the power set of $\aleph_{\alpha}$ is also a set in $M_{\infty}$.
\end{proof}
\begin{defn}
A partially order class $\mathbb{R}$ is a \emph{chain of complete
subposets} if $\mathbb{R}=\bigcup\limits _{\alpha\in\mathbf{Ord}}\mathbb{R}_{\alpha}$,
where each $\mathbb{R}_{\alpha}$ is a partially ordered set, such
that if $\alpha\leq\beta$ then $\mathbb{R}_{\alpha}$ is a complete
suborder of $\mathbb{R}_{\beta}$.
\end{defn}
\begin{lem}
\label{lem:completeSubposets}$\mathbb{P}$ is a chain of complete
subposets.
\end{lem}
\begin{proof}
By definition \ref{def:class-iter} $\mathbb{P}=\bigcup\limits _{\alpha\in\mathbf{Ord}}\mathbb{P}_{\alpha}$.This
is an iterated forcing, and so the identity map $i_{\alpha,\beta}:\mathbb{P}_{\alpha}\rightarrow\mathbb{P}_{\beta}$
is a complete embedding (see ch. VIII lemma 5.11 in \citep{key-19}).
Therefore $\mathbb{P}_{\alpha}$ is a complete suborder of $\mathbb{P}_{\beta}$.
\end{proof}
\begin{defn}
$\mathbb{P}=\bigcup\limits _{\alpha\in\mathbf{Ord}}\mathbb{P}_{\alpha}$
is a \emph{progressively closed iteration} if $\mathbb{P}$ is a chain
of complete subposets, and for arbitrarily large regular cardinals
$\delta$ there are arbitrarily large $\alpha$ such that there is
a $\mathbb{P}_{\alpha}$-name $\dot{\mathbb{P}}_{\left[\alpha,\infty\right)}=\left\{ \left\langle op\left(\check{\beta},\mathbb{\dot{P}}_{\left[\alpha,\beta\right)}\right),0\right\rangle \mid\beta>\alpha\right\} $
satisfying:
\begin{enumerate}
\item For every $\beta>\alpha$ the poset $\mathbb{P}_{\beta}$ is isomorphic
to the two-stage iteration $\mathbb{P}_{\beta}\cong\mathbb{P}_{\alpha}\star\mathbb{\dot{P}}_{\left[\alpha,\beta\right)}$;
\item $\mathbb{P}_{\alpha}\Vdash\check{\delta}\text{ is a regular cardinal and }\mathbb{\dot{P}}_{\left[\alpha,\beta\right)}\text{ is }<\check{\delta}\text{-closed}$;
\item For $\beta'>\beta>\alpha$ the isomorphisms at $\beta$ and $\beta'$
yield complete subposets $\mathbb{P}_{\alpha}\star\mathbb{\dot{P}}_{\left[\alpha,\beta\right)}\subseteq_{c}\mathbb{P}_{\alpha}\star\mathbb{\dot{P}}_{\left[\alpha,\beta'\right)}$
such that the complete embeddings commute with the isomorphisms.
\item $\mathbb{P}_{\alpha}\Vdash\dot{\mathbb{P}}_{\left[\alpha,\infty\right)}\text{ is a chain of complete subposets}$.
\end{enumerate}
\end{defn}
\begin{lem}
$\mathbb{P}$ is a progressively closed iteration.
\end{lem}
\begin{proof}
By lemma \ref{lem:completeSubposets} $\mathbb{P}$ is a chain of
complete subposets. Let $\aleph_{\delta}$ be a successor cardinal,
and $\alpha=\delta+1$.
\begin{enumerate}
\item By definition \ref{def:class-iter} $\mathbb{P}_{\beta}=\mathbb{P}_{\alpha}\star\mathbb{\dot{P}}_{\left[\alpha,\beta\right)}$.
\item By lemma \ref{lem:ClassSameCard} all the cardinals are preserved,
and by lemma \ref{lem:ClassRelativeClosure} $\mathbb{\dot{P}}_{\left[\alpha,\beta\right)}$
is $\aleph_{\delta+1}$-closed.
\item Let $p\in\mathbb{P}_{\beta}$. Then $p=\left\langle p^{<\alpha},\dot{p}^{\alpha},\dot{p}^{\alpha+1},...,\dot{p}^{\gamma},...\mid\gamma<\beta\right\rangle $.
Which embeds to $\left\langle p^{<\alpha},\dot{p}^{\alpha},\dot{p}^{\alpha+1},...,\dot{p}^{\gamma},...,0,...\mid\gamma<\beta\right\rangle \in\mathbb{P}_{\alpha}\star\mathbb{\dot{P}}_{\left[\alpha,\beta'\right)}$.
Similarly $p$ embeds to $\left\langle p^{<\beta},0,...\right\rangle \in\mathbb{P}_{\beta'}$,
which through the isomorphism is equal to $\left\langle p^{<\alpha},\dot{p}^{\alpha},\dot{p}^{\alpha+1},...,\dot{p}^{\gamma},...,0,...\mid\gamma<\beta\right\rangle \in\mathbb{P}_{\alpha}\star\mathbb{\dot{P}}_{\left[\alpha,\beta'\right)}$.
Hence the complete embeddings commute with the isomorphisms.
\item Lemma \ref{lem:completeSubposets} applies to the tail of the forcing
as well.
\end{enumerate}
\end{proof}
\begin{lem}
$M_{0}\left[\mathbb{G}\right]\vDash\mathsf{BGC}$ and $M_{\infty}\vDash\mathsf{ZFC}$.
\end{lem}
\begin{proof}
By theorem 98 of \citep{key-2} a progressively closed iteration generates
a generic extension that satisfies $\mathsf{BGC}$.

So $M_{0}\left[\mathbb{G}\right]\vDash\mathsf{BGC}$ and by fact \ref{fact:BG->ZF}
$M_{\infty}\vDash\mathsf{ZFC}$.
\end{proof}
\begin{lem}
\label{lem:classIsUnion}$M_{\infty}=\bigcup\limits _{\alpha\in\mathbf{Ord}}M_{\alpha}$.
\end{lem}
\begin{proof}
According to lemma \ref{lem:completeSubposets} $\mathbb{P}$ is a
chain of complete subposets. So applying lemma 88 of \citep{key-2}
to the sets, we get $M_{\infty}=\bigcup\limits _{\alpha\in\mathbf{Ord}}M_{0}\left[G_{\alpha}\right]$.
\end{proof}
\begin{thm}
\label{thm:classWellOrdered}$N$ is a proper inner model of $M_{\infty}$
if and only if for some $\alpha\in\mathbf{Ord}$ $N=M_{\alpha}$.
\end{thm}
\begin{proof}
Working to the contrary, assume there exists $N$ an inner model of
$M_{\infty}$ such that $N\neq M_{\alpha}$ for all $\alpha\in\mathbf{Ord}$.

Suppose there exists a greatest ordinal $\beta$ such that $M_{\beta}\subseteq N$.
$M_{\beta}\vDash\mathsf{AC}$ so according to Vop\v{e}nka \citep{key-20}
there exists a set of ordinals $A\in N\setminus M_{\beta}$. By lemma
\ref{lem:classIsUnion} there exists an ordinal $\alpha$ such that
$A\in M_{\alpha}$. But that means $M_{\beta}\subsetneq M_{\beta}\left[A\right]\subseteq M_{\alpha}$,
so according to theorem \ref{thm:chainOfModels} $M_{\beta}\left[A\right]=M_{\gamma}$
for some $\beta<\gamma\leq\alpha$.

Therefore $M_{\beta}\left[A\right]=M_{\gamma}\subseteq N$, in contradiction
to $\beta$ being the greatest ordinal such that $M_{\beta}\subseteq N$.
So there is no such greatest $\beta$.

Next, suppose that for a limit ordinal $\delta$, $M_{\alpha}\subsetneq N$
for all $\alpha<\delta$.

As shown in the proof of lemma \ref{lem:lim_case}, there is a set
$g=\left\langle g_{\alpha}\mid\alpha<\delta\right\rangle \in N$ such
that if $\alpha$ is not a limit ordinal $g_{\alpha}\in M_{\alpha+1}\setminus M_{\alpha}$
and $g_{\alpha}\subseteq\aleph_{\alpha}$, and if $\alpha$ is a limit
ordinal then $g_{\alpha}=\emptyset$. Define $g'_{\alpha}=\left\{ \omega_{\alpha}+\beta\mid\beta\in g_{\alpha}\right\} $.
Obviously $g'_{\alpha}\subseteq\omega_{\alpha+1}\setminus\omega_{\alpha}$.
Define $g'=\bigcup\limits _{\alpha<\delta}g'_{\alpha}$.

We get $g'\in N$ a set of ordinals. If $g'\in M_{\alpha}$ for some
$\alpha<\delta$ then $g'_{\alpha}=\left(g'\cap\omega_{\alpha+1}\right)\setminus\omega_{\alpha}\in M_{\alpha}$
and then $g_{\alpha}\in M_{\alpha}$ in contradiction to its definition.
Therefore $g\notin M_{\alpha}$ for all $\alpha<\delta$.

By lemma \ref{lem:classIsUnion} there exists an ordinal $\beta$
such that $g'\in M_{\beta}$, so $M_{0}\left[g'\right]\subseteq M_{\beta}$.
According to theorem \ref{thm:chainOfModels} this means $M_{0}\left[g'\right]=M_{\gamma}$
for some $\delta\leq\gamma\leq\beta$, and so $M_{0}\left[g'\right]\supseteq M_{\delta}$.

On the other hand, because $g'$ is a set of ordinals and $M_{0}\vDash\mathsf{AC}$,
$M_{0}\left[g'\right]$ is the smallest model of $\mathsf{ZFC}$ such
that $g'\in M_{0}\left[g'\right]$, and so $M_{0}\left[g'\right]\subseteq N$.
Hence $M_{\delta}\subseteq N$.

As a result, by induction for all $\alpha\in\mathbf{Ord}$ $M_{\alpha}\subseteq N$,
and so $M_{\infty}=\bigcup\limits _{\alpha\in\mathbf{Ord}}M_{\alpha}\subseteq N$,
in contradiction of $N$ being a proper inner model of $M_{\infty}$.

We conclude that $N$ is a proper inner model of $M_{\infty}$ if
and only if $N=M_{\alpha}$ for some $\alpha\in\mathbf{Ord}$.
\end{proof}
At last, we arrive at what we set out to prove:
\begin{cor}
The existence of well-ordered model universes with the height of the
ordinals is consistent with $\mathsf{ZFC}$.
\end{cor}
\begin{proof}
Theorem \ref{thm:classWellOrdered} shows that $M_{\infty}$ is a
well-ordered model universe with $ht\left(M_{\infty}\right)=\infty$.
\end{proof}
We conclude this section with the observation that in $M_{\infty}$
the class of inner models $\mathbb{M}\left(M_{\infty}\right)$ (as
defined in \ref{def:Well-Ordered-Model}) is in fact definable in
$M_{\infty}$: $M_{0}=L$, for all $\alpha$ $M_{\alpha+1}=L\left(\mathcal{P}\left(\aleph_{\alpha}\right)^{M_{\infty}}\right)$,
and for all limit $\delta$ $M_{\delta}=L\left(\mathcal{P}\left(\aleph_{\delta}\right)^{M_{\infty}}\right)$.
Thus $M_{\infty}$ in a sense 'knows' that it is a well-ordered model
universe.

\section{Open questions}

In the previous section we constructed an example of a nice well-ordered
model universe of height equal to $\mathbf{Ord}$. We did this by
an iteration of progressively increasing $\kappa$-Sacks forcing.
In this section we discuss some remaining open questions regarding
well-ordered model universes:
\begin{enumerate}
\item Can we construct a well-ordered model universe that isn't nice?
\item What can we say about models when the inner models are just totally-ordered,
not well-ordered by inclusion?
\item What if we consider all inner models of $\mathsf{ZF}$, not just inner
models of $\mathsf{ZFC}$?
\end{enumerate}

\subsection{Non-nice well-ordered model universes}

For the first question, recall definition \ref{def:niceness}. A well-ordered
model universe is considered nice if its underlying order is equivalent
to some ordinal or to $\mathbf{Ord}$. This is essentially a limit
on the length of the well-ordering. Any well-ordered set is order-isomorphic
to some ordinal, so if a well-ordered model universe isn't nice then
the underlying order must be a proper class, but one which is not
order-isomorphic to $\mathbf{Ord}$.

Can we define such a well-ordering? Of course - just take $A=\left\{ \alpha\mid\alpha\in\mathbf{Ord}\vee\alpha=\left\{ 1\right\} \right\} $,
and extend the natural ordering by defining $\left\{ 1\right\} >\alpha$
for all $\alpha\in\mathbf{Ord}$. It is easy to see that this is indeed
a well-ordering: if $B\subseteq A$ is a non-empty class, then if
it contains any ordinal, then the least ordinal it contains is its
least element according to our extended ordering, and if not then
$\left\{ 1\right\} $ is the least element. It is also obvious that
our extended ordering is not order-isomorphic to $\mathbf{Ord}$ -
our ordering has a greatest element, whereas $\mathbf{Ord}$ clearly
does not.

So such a well-ordering is very much definable. Could we extend our
construction further then we did in the previous section?

For the rest of the subsection, let $M_{\infty}$ be as defined in
\ref{def:class-notation}. In general, there is no obstacle to applying
$\aleph_{\beta}$-Sacks forcing to $M_{\infty}$. The normal properties
of Sacks-forcing would still hold, i.e there won't be any intermediate
model between $M_{\infty}$ and $M_{\infty}\left[G\right]$. Moreover,
there would be a chain of inner models of $M_{\infty}\left[G\right]$
that would be 'longer' than $\mathbf{Ord}$. However, regardless of
the $\aleph_{\beta}$-Sacks forcing we use, $M_{\infty}\left[G\right]$
would invariably contain some new inner model that is not on the chain.
\begin{lem}
\label{lem:NotInfty+1}Let $\aleph_{\beta}$ be a regular cardinal,
let $\mathbb{S}$ be the $\aleph_{\beta}$-Sacks forcing notion over
$M_{\infty}$, and let $H$ be a generic set in $\mathbb{S}$. Then
$M_{\infty}\left[H\right]$ is not a well-ordered model universe.
\end{lem}
\begin{proof}
By lemma \ref{lem:classIsUnion} $M_{\infty}=\bigcup\limits _{\alpha\in\mathbf{Ord}}$.
Therefore $\mathcal{P}\left(\mathbb{S}\right)\in M_{\alpha}$ for
some $\alpha\in\mathbf{Ord}$. As all the dense sets of $\mathbb{S}$
in $M_{\infty}$ are already in $M_{\alpha}$, $H$ is also a generic
set of $\mathbb{S}\in M_{\alpha}$, and so $M_{\alpha}\left[H\right]$
is a generic extension generated by $\aleph_{\beta}$-Sacks forcing
over $M_{\alpha}$.

Obviously $M_{\alpha}\left[H\right]\subseteq M_{\infty}\left[H\right]$,
but $M_{\alpha}\left[H\right]\nsubseteq M_{\infty}$. So for all $\gamma\in\mathbf{Ord}$
$M_{\alpha}\left[H\right]\neq M_{\gamma}$ and $M_{\alpha}\left[H\right]\neq M_{\infty}$.

By theorem \ref{thm:-Sacks-forcing-produces} there are no intermediate
models between $M_{\infty}$ and $M_{\infty}\left[H\right]$. So either
$M_{\alpha}\left[H\right]=M_{\infty}\left[H\right]$, or $M_{\alpha}\left[H\right]$
is 'off-chain'.

But $M_{\alpha}\left[H\right]$ is $\aleph_{\beta}$-Sacks forcing
over $M_{\alpha}$, and so has no intermediate model between $M_{\alpha}$
and $M_{\alpha}\left[H\right]$, whereas $M_{\alpha}\subsetneq M_{\infty}\subsetneq M_{\infty}\left[H\right]$.
Therefore $M_{\alpha}\left[H\right]\neq M_{\infty}\left[H\right]$,
so $M_{\alpha}\left[H\right]$ is 'off-chain', and $M_{\infty}\left[H\right]$
is not a well-ordered model universe. 
\end{proof}
Okay, but lemma \ref{lem:NotInfty+1} only shows that we can't use
$\kappa$-Sacks forcing to produce the next step of the construction.
Could some other set forcing notion do the trick for us?

Looking back at theorem \ref{thm:chainOfModels}, we proved that we
could take \emph{any} set created by the iteration and use it to completely
recover all the preceding generic sets. So essentially, each generic
set must code all the preceding generic sets. But because we used
class-many generic sets to construct $M_{\infty}$, we need our new
generic set to encode class-many previous generic sets, which is a
tall order. In fact, it is impossible:
\begin{lem}
\label{lem:setIsntWell}Let $\mathbb{S}$ be some minimal set forcing
notion over $M_{\infty}$, and let $H$ be a generic set in $\mathbb{S}$.
Then $M_{\infty}\left[H\right]$ is not a well-ordered model universe.
\end{lem}
\begin{proof}
Let $\aleph_{\beta}=\left|\mathbb{S}\right|$. By lemma \ref{lem:classIsUnion}
there exists an ordinal $\alpha$ such that $\mathcal{P}\left(\mathbb{S}\right)\in M_{\alpha}$.
Take $\gamma=\max\left(\alpha,\beta+1\right)$.

Obviously $\mathbb{S}$ has the $\aleph_{\beta+1}$-c.c property.
By lemma \ref{lem:ClassRelativeClosure} $\nicefrac{\mathbb{P}}{G_{\gamma}}$
is at least $\aleph_{\beta+1}$-closed, and by lemma \ref{lem:completeSubposets}
$\nicefrac{\mathbb{P}}{G_{\gamma}}$ is a chain of complete subposets.

Because $\mathbb{S}\in M_{\gamma}$, $M_{\infty}\left[H\right]$ is
actually the result of product forcing, where the first forcing is
the tail of class forcing $\mathbb{P}$ and the second set forcing
$\mathbb{S}$, so $M_{\infty}\left[H\right]=M_{\gamma}\left[\nicefrac{\mathbb{G}}{G_{\gamma}}\right]\left[H\right]$.

By lemma 121 in \citep{key-2} we have that $\nicefrac{\mathbb{G}}{G_{\gamma}}$
is $\nicefrac{\mathbb{P}}{G_{\gamma}}$-generic over $M_{\gamma}\left[H\right]$.
Therefore $M_{\gamma}\left[H\right]\neq M_{\infty}\left[H\right]$.

Also, because $H\in M_{\gamma}\left[H\right]$ we have $M_{\gamma}\left[H\right]\nsubseteq M_{\infty}$.
And because $\mathbb{S}$ is minimal, so there is no intermediate
model between $M_{\infty}$ and $M_{\infty}\left[H\right]$.

Therefore $M_{\gamma}\left[H\right]$ is a proper inner model of $M_{\infty}\left[H\right]$
that is off the chain, so $M_{\infty}\left[H\right]$ is not a well-ordered
model universe.
\end{proof}
And what about class forcing? Could that be used to somehow lengthen
our well-ordered model universe?

For class forcing of chain of complete subposets, this is again impossible.
\begin{lem}
Let $\mathbb{S}$ be a chain of complete subposets over $M_{\infty}$,
and let $\text{\ensuremath{\mathbb{H}}}$ be a generic class in $\mathbb{S}$
such that $M_{\infty}\left[\mathbb{H}\right]\vDash\mathsf{ZFC}$.
Then $M_{\infty}\left[\mathbb{H}\right]$ is not a well-ordered model
universe.
\end{lem}
\begin{proof}
According to lemma 88 of \citep{key-2} $M_{\infty}\left[\mathbb{H}\right]=\bigcup\limits _{\alpha\in\mathbf{Ord}}M_{\infty}\left[H_{\alpha}\right]$,
so $A\in M_{\infty}\left[H_{\alpha}\right]$ for some $\alpha\in\mathbf{Ord}$.

But by lemma \ref{lem:setIsntWell} $M_{\infty}\left[H_{\alpha}\right]$
is not a well-ordered model universe, so the inner models of $M_{\infty}\left[H_{\alpha}\right]$
are not well-ordered by inclusion.

As $M_{\infty}\left[H_{\alpha}\right]$ is definable with set parameters
in $M_{\infty}\left[\mathbb{H}\right]$, so are the inner models of
$M_{\infty}\left[H_{\alpha}\right]$ similarly definable, and so they
are inner models of $M_{\infty}\left[\mathbb{H}\right]$.

Therefore the inner models of $M_{\infty}\left[\mathbb{H}\right]$
are not well-ordered by inclusion.
\end{proof}
What about some more general form of class forcing?

At first thought this might also appear impossible, because even for
class forcing to minimally extend $M_{\infty}$, we would still need
\emph{every} new \emph{set} in $M_{\infty}\left[\mathbb{H}\right]$
to somehow encode the entire \emph{class} of generic sets! However,
the remarkable Jensen's Coding Theorem \citep{key-33} actually uses
class forcing to achieve something similar: the existence of class
forcing notion $\mathbb{P}$ such that if $\mathbb{G}$ is $\mathbb{P}$-generic
over $V$ then $V\left[\mathbb{G}\right]\vDash\mathsf{ZFC}+V=L\left[A\right]+A\subseteq\omega$.
This set $A$ in effect 'codes the universe'. Applying the theorem
to $M_{\infty}$, all we really need is for every new set in $M_{\infty}\left[\mathbb{H}\right]$
to code this $A$, which sounds far more reasonable. So we are left
with the following open question:
\begin{problem}
Is the existence of a well-ordered model universe with an underlying
order longer than $\mathbf{Ord}$ consistent with $\mathsf{ZFC}$?
\end{problem}

\subsection{Totally-ordered model universes}

So far in this article we focused exclusively on models where all
the inner models are well-ordered by inclusion. However, a natural
weakening of the definition is to demand the ordering to only be total,
i.e for any two inner models of $V$, either $M_{1}\subseteq M_{2}$
or $M_{2}\subseteq M_{1}$. We'll call this a \emph{totally-ordered
model universe}.

Prima facie, this concept is far weaker than a well-ordered model
universe. For one, our proof that $V$ has no measurable cardinals
in theorem \ref{thm:Measurable} immediately fails, because in theory
there could be an infinite descending sequence of inner models. However,
upon closer inspection we find that the proof of theorem \ref{thm:.0=000023}
actually still holds, because it hinges on the fact if both $G_{0}$
is generic over $L\left[G_{1}\right]$ and vice-versa, then the inner
models aren't totally-ordered. Therefore:
\begin{thm}
If $V$ is a totally-ordered model universe, $0^{\sharp}$ doesn't
exist.
\end{thm}
\begin{proof}
Identical to theorem \ref{thm:.0=000023}.
\end{proof}
\begin{cor}
If $V$ is a totally-ordered model universe, then $V$ has no measurable
cardinal.
\end{cor}
\begin{proof}
By Gaifman \citep{key-31}, the existence of a measurable cardinal
implies the existence of $0^{\sharp}$.
\end{proof}
So some of the basic properties of well-ordered model universes extend
to totally-ordered model universes, and the total-ordering property
by itself is sufficient to prove that $V$ is inherently small and
quite 'close' to $L$.

Therefore for each result proven about well-ordered model universes,
we should ask ourselves whether it extends to totally-ordered model
universes as well.

\subsection{Inner models of $\mathsf{ZF}$}

Another natural extension of the definition of a well-ordered model
universe is the consideration of general inner models of $\mathsf{ZF}$,
not just inner models that satisfy Choice.

Returning to the framework we introduced at the beginning, we give
the following expanded definition, which is almost verbatim definition
\ref{def:Well-Ordered-Model}:
\begin{defn}
\label{def:Well-Ordered-Model-1}Let $\left\langle V,\mathcal{V},\in\right\rangle $
be a model of $\mathsf{BG}$. We call a model $N\subseteq V$ of $\mathsf{ZF}$
where all its inner $\mathsf{ZF}$ models are well-ordered with respect
to inclusion a \emph{well-ordered $\mathsf{ZF}$ model universe}.
Formally, we postulate the existence of a class $\mathbb{M}$ in $\left\langle V,\mathcal{V},\in\right\rangle $,
which is the sequence of all proper inner $\mathsf{ZF}$ models of
$N$ ordered by inclusion. This means:
\begin{enumerate}
\item $\mathbb{M}\subseteq I\times N$;
\item \label{enu:Every-Inner-1}$M$ is a proper inner $\mathsf{ZF}$ model
of $N$ if and only if there exists a unique $a\in I$ such that $M=M_{a}=\left\{ x\mid\left(a,x\right)\in\mathbb{M}\right\} $;
\item $I$ is a well-ordered class;
\item If $a<_{I}b$ then $\left(a,x\right)\in\mathbb{M}\rightarrow\left(b,x\right)\in\mathbb{M}$.
\end{enumerate}
\end{defn}
In summary, applying the convention that lower-case letters indicate
sets and upper-case letters indicates classes, we demand the following
be true: 
\begin{multline*}
\exists\mathbb{M}\exists I\text{ (}\mathbb{M}\subseteq I\times N\wedge\left(M\subsetneq N\text{ is an inner \ensuremath{\mathsf{ZF}} model }\leftrightarrow\exists!a\in I\left(M=\left\{ x\mid\left(a,x\right)\in\mathbb{M}\right\} \right)\right)\wedge\\
I\text{ is well-ordered }\wedge a<_{I}b\rightarrow\left(\left(a,x\right)\in\mathbb{M}\rightarrow\left(b,x\right)\in\mathbb{M}\right)\text{)}
\end{multline*}

We define the height of $V$ in exactly the same way we did for the
original definition, using the order type of $I$.

Obviously, if $N\vDash\mathsf{ZFC}$ is a well-ordered $\mathsf{ZF}$
model universe then it is also a well-ordered model universe. Given
the sequence of proper $\mathsf{ZF}$ inner models we can directly
define the sequence of proper $\mathsf{ZFC}$ inner models as defined
in \ref{def:Well-Ordered-Model}. However, by our definition well-ordered
$\mathsf{ZF}$ model universes are not required themselves to satisfy
$\mathsf{AC}$, and therefore not every well-ordered $\mathsf{ZF}$
model universe is necessarily a well-ordered model universe.

Next we outline a few of the basic properties of well-ordered $\mathsf{ZF}$
model universes. For the rest of the subsection, assume $\left\langle V,\mathcal{V},\in\right\rangle \vDash\mathsf{BG}$,
$V$ is a well-ordered $\mathsf{ZF}$ model universe, and $\mathbb{M}\subseteq I\times V$
is its sequence of proper inner $\mathsf{ZF}$ models ordered by inclusion.
\begin{lem}
$M_{0}=L$
\end{lem}
\begin{proof}
Identical to lemma \ref{lem:M0=00003DL}.
\end{proof}
\begin{thm}
$V\models\textrm{There is no measurable cardinal}$.
\end{thm}
\begin{proof}
Identical to theorem \ref{thm:Measurable}.
\end{proof}
The proof of theorem \ref{thm:L=00005BA=00005D} doesn't work for
well-ordered $\mathsf{ZF}$ model universes, as it involves heavy
use of the Axiom of Choice. In general, it is very much possible to
have an infinite chain of inner models that satisfy $\mathsf{AC}$,
but that the least inner model to include them all does not. So even
though we can carry out the successor stage of the proof, we cannot
prove $M_{\omega}\vDash\mathsf{AC}$, meaning it is very possible
$M_{\omega}\neq L\left[A\right]$ for all $A\in V$.

Note that because we can still carry out the successor stages of the
induction, we have the following corollary:
\begin{cor}
\label{cor:height}If $V\vDash\neg\mathsf{AC}$ then $ht(V)\geq\omega$.
\end{cor}
\begin{proof}
Applying the successor steps in the proof of theorem \ref{thm:L=00005BA=00005D},
we get that for all $n<\omega$ $M_{n}=L\left[A\right]$ for some
$A\in V$.
\end{proof}
This result can actually be strengthened, even without the well-ordering
property:
\begin{thm}
\label{thm:InfiniteModels}If $V$ is not of the form $L\left[A\right]$
for some $A\in V$, then $V$ has an infinite number of inner models.
\end{thm}
\begin{proof}
By induction we prove every model with a finite number of inner models
is of the form $L\left[A\right]$. The case for $0$ proper inner
models is trivially true because $L=L\left[\emptyset\right]$.

Now assume that we've proven the induction for models with $n$ proper
inner models. Assume $K$ is a model that has $n+1$ proper inner
models. Every inner model of an inner model of $K$ is an inner model
of $K$ itself, so all inner models of $K$ have at most $n$ proper
inner models, and so they are all of the form $L\left[A_{m}\right]$
for some $m\leq n$.

Let's consider two possibilities: either $K$ has a greatest proper
inner model $R\subsetneq K$, or it doesn't. If $R$ exists, then
by Vop\v{e}nka \citep{key-20} there is a set of ordinals $A\in K\setminus R$,
and so $K=L\left[A\right]$ as required.

Otherwise, just as we did in the proof of theorem \ref{thm:L=00005BA=00005D},
we can arrange a family of mutually disjoint sets of ordinals $B_{m}$
such that for all $m$ $L\left[A_{m}\right]=L\left[B_{m}\right]$.
Take $L\left[\bigcup\limits _{m\leq n}B_{m}\right]$. For all $m$
$L\left[\bigcup\limits _{m\leq n}B_{m}\right]\supseteq L\left[B_{m}\right]$.
But the only model that includes all proper inner models of $K$ is
$K$ itself. Therefore $L\left[\bigcup\limits _{\alpha}B_{\alpha}\right]=K$.

Therefore if $V$ is not a model of the form $L\left[A\right]$, $V$
must have an infinite number of inner models.
\end{proof}
\begin{cor}
\label{cor:finiteHeight}If $V$ is a well-ordered $\mathsf{ZF}$
model universe and $ht\left(V\right)<\omega$ then $V$ is a well-ordered
model universe.
\end{cor}
\begin{proof}
According to theorem \ref{thm:InfiniteModels} $V$ and all of its
inner models must satisfy the Axiom of Choice. Therefore by definition
\ref{def:Well-Ordered-Model-1} $V$ is a well-ordered model universe.
\end{proof}
Despite the failure of theorem \ref{thm:L=00005BA=00005D} for well-ordered
$\mathsf{ZF}$ model universes, we have a small consolation prize:
\begin{lem}
\label{lem:L(A) successors}Let $\alpha+1\in I$ denote the successor
of $\alpha$ in the well-ordering of $I$. Then $M_{\alpha+1}=L\left(A\right)$
for some $A\in V$.
\end{lem}
\begin{proof}
Take $A\in M_{\alpha+1}\setminus M_{\alpha}$. $L\left(A\right)$
is the smallest inner $\mathsf{ZF}$ model containing $A$, so obviously
$L\left(A\right)\subseteq M_{\alpha+1}$. But because $V$ is a well-ordered
universe $M_{\alpha}\subsetneq L\left(A\right)$, and there are no
intermediate inner models between $M_{\alpha}$ and $M_{\alpha+1}$.
Therefore $L\left(A\right)=M_{\alpha+1}$.
\end{proof}
\begin{cor}
If $V$ has a greatest proper inner $\mathsf{ZF}$ model $K$, there
exists $A\in V$ such that $V=L\left(A\right)$.
\end{cor}
\begin{proof}
Take $A\in V\setminus K$. We get $K\subsetneq L\left(A\right)\subseteq V$,
and so for the same reasons as lemma \ref{lem:L(A) successors} $L\left(A\right)=V$.
\end{proof}
In conclusion, there isn't much we know about well-ordered $\mathsf{ZF}$
model universes.

As for actually constructing a well-ordered $\mathsf{ZF}$ model universe,
corollary \ref{cor:finiteHeight} shows that using the iteration defined
in \ref{def:iteration} up to finite height, would generate a well-ordered
$\mathsf{ZF}$ model universe.

A well-ordered $\mathsf{ZF}$ model universe of height $\omega$ is
achievable by iterating the forcing up to $M_{\omega}$, and then
taking $N=\mathsf{HOD}\left(\left\langle G_{n}\mid n<\omega\right\rangle \right)\subsetneq M_{\omega}$.
This $N$ will be the minimal inner model of $\mathsf{ZF}$ that includes
$M_{n}$ for all $n<\omega$.

However, we can't use the same construction to build well-ordered
$\mathsf{ZF}$ model universes of arbitrary height, because we can't
tell what's going on between $N$ and $M_{\omega}$. The intermediate
inner models there might not even be totally-ordered.

So we are left with one glaring open question:
\begin{problem}
Is the existence of a well-ordered $\mathsf{ZF}$ model universe the
height of the ordinals consistent with $\mathsf{ZF}$?
\end{problem}

\end{document}